\documentclass[11pt]{amsart}
\usepackage[utf8]{inputenc}
\usepackage{tabulary}
\usepackage{amssymb}
\usepackage{tikz}
\usepackage[backend=biber,style=alphabetic,maxbibnames=99]{biblatex}
\usepackage{multirow}
\usepackage{textcomp}
\usepackage{csquotes}
\usepackage{listings}
\usepackage{hyperref}
\usepackage{xcolor}
\hypersetup{colorlinks=false}
\DeclareUnicodeCharacter{0229}{\c{e}}
\DeclareUnicodeCharacter{0159}{\v{r}}
\DeclareUnicodeCharacter{0142}{\l}
\DeclareUnicodeCharacter{015A}{\'{S}}
\DeclareUnicodeCharacter{015B}{\'{s}}
\newtheorem{theorem}{Theorem}[section]

\newtheorem{corollary}[theorem]{Corollary}
\newtheorem{proposition}[theorem]{Proposition}
\newtheorem{lemma}[theorem]{Lemma}

\theoremstyle{definition}
\newtheorem{definition}[theorem]{Definition}

\newtheorem{remark}[theorem]{Remark}


\title{Lipschitz retractions and complementation properties of Banach spaces}
\author{Petr H\'ajek}\thanks{This research was supported by CAAS CZ.02.1.01/0.0/0.0/16-019/0000778 and by the project SGS21/056/OHK3/1T/13.}
\address[P. H\'ajek]{Czech Technical University in Prague, Faculty of Electrical Engineering.
Department of Mathematics, Technick\'a 2, 166 27 Praha 6 (Czech Republic)}
\email{hajek@math.cas.cz}

\author{Andr\'es Quilis}\thanks{The second author's research has been supported by PAID-01-19}
\address[A. Quilis]{Universitat Polit\`ecnica de Val\`encia. Instituto Universitario de Matem\'atica Pura y Aplicada, Camino de Vera, s/n
46022 Valencia (Spain); and Czech Technical University in Prague, Faulty of Electrical Engineering. Department of Mathematics, Technick\'a 2, 166 27 Praha 6 (Czech Republic)}
\email{anquisan@posgrado.upv.es}
\begin{document}
\maketitle

\begin{abstract}
In the present paper we introduce and study the Lipschitz retractional structure of metric spaces. This topic was motivated by the analogous projectional structure of Banach spaces, a topic that has been thoroughly investigated. The more general metric setting fits well with the currently active theory of Lipschitz free spaces and spaces of Lipschitz functions. Among our applications we show that the Lipschitz free space $\mathcal{F}(X)$ is a Plichko space whenever $X$ is a Plichko Banach space. Our main results include two examples of metric spaces. The first one $M$  contains two points $\{0,1\}$ such that  no separable subset of $M$ containing these points is a Lipschitz retract of $M$. The second example fails the analogous property for arbitrary infinite  density. Finally, we introduce the metric version of the concept of locally complemented Banach subspace, and prove some metric analogues to the linear theory.
\end{abstract}

\section{Introduction}
In this note we are interested in studying Lipschitz analogues to the Separable Complementation Property (SCP) of Banach spaces and the related concepts.
Most of our results are formulated in terms of various retractional properties of metric spaces, but there is a close relationship with the structural theory of Lipschitz free spaces over Banach (or metric)  spaces, which was our initial motivation. In this direction, our results indicate that the Lipschitz free spaces of general nonseparable Banach spaces have a much richer
complementation structure than their source Banach spaces. 

 Let us start by giving some definitions. A pointed metric space $(M,d,0)$ is a metric space $(M,d)$ with a selected  distinguished point $0$ in $M$, called the base point. We will often write just $M$ to refer to a pointed metric space $(M,d,0)$. For a pointed metric space $M$, we consider the Banach space  $\text{Lip}_0(M)$  formed by all real-valued Lipschitz functions that vanish at the base point, endowed with the norm given by Lipschitz constant
$$\|f\|_\text{Lip}=\sup\bigg\{\frac{|f(p)-f(q)|}{d(p,q)}\colon p,q\in M,~p\neq q\bigg\},$$
for every $f\in \text{Lip}_0(M)$. 
The Banach space $\text{Lip}_0(M)$ is a dual space. Its canonical predual is the Lipschitz free space associated to $M$, denoted by $\mathcal{F}(M)$. It is given by $\mathcal{F}(M)=\overline{\text{span}}\{\delta(p)\colon p\in M\}\subset\text{Lip}_0(M)^{**}$, where $\delta(p)\colon \text{Lip}_0(M)\rightarrow\mathbb{R}$ is the linear functional defined by $\delta(p)(f)=f(p)$ for all $f\in\text{Lip}_0(M)$. An important property of Lipschitz free spaces is their universal linearization of Lipschitz functions. More precisely, for any Lipschitz map $F\colon M\rightarrow N$ between two pointed metric spaces such that $F(0)=0$, there exists a linear and bounded operator $\widehat{F}\colon \mathcal{F}(M)\rightarrow\mathcal{F}(N)$ such that $\|\widehat{F}\|=\|F\|_\text{Lip}$, and $\widehat{F}\circ\delta_M=\delta_N\circ F$. The map $\delta\colon M\rightarrow\mathcal{F}(M)$ is a nonlinear isometric embedding of a metric space into its Lipschitz free space. If $X$ is a Banach space, this isometric embedding has a linear left inverse $\beta\colon \mathcal{F}(X)\rightarrow X$, obtained by extending the barycenter map of finitely supported measures on $X$, to the whole Lipschitz free space. We refer to the important paper by Godefroy and Kalton \cite{GodKal03} for a discussion of these concepts, as well as the monograph by Weaver \cite{Wea18Book}, where metric spaces and Lipschitz functions are studied in great detail. 

Since we intend to study SCP analogues in the context of metric spaces, let us first introduce and summarize the situation in the Banach space setting:

Let $X$ be a Banach space and let $Y$ be a closed linear subspace of $X$. We say that:

\begin{itemize}
    \item[(1)] $Y$ is \emph{linearly complemented in $X$} if there exists a bounded linear map $P\colon X\rightarrow Y$ such that $Py=y$ for all $y\in Y$. We call $P$ a \emph{linear projection from $X$ onto $Y$}.
    \item[(2)] $Y$ is a \emph{Lipschitz retract of $X$} if there exists a Lipschitz map $R\colon X\rightarrow Y$ such that $R(y)=y$ for all $y\in Y$. We call $R$ a \emph{Lipschitz retraction from $X$ onto $Y$}.
    \item[(3)] $Y$ is an \emph{$\mathcal{F}$-Lipschitz retract of $X$} if there exists a Lipschitz map $F\colon X\rightarrow \mathcal{F}(Y)$ such that $F(y)=\delta(y)$ for all $y\in Y$. We call $F$ an \emph{$\mathcal{F}$-Lipschitz retraction from $X$ onto $Y$}. 
    \item[(4)] $Y$ is {locally complemented in $X$} if for every $E$ finite-dimensional subspace of $X$ there exists a linear map $T\colon E\rightarrow Y$ such that $Tf=f$ for all $f\in F\cap Y$.  
\end{itemize}
Observe that $(3)$ is equivalent to $\mathcal{F}(Y)$ being linearly complemented in $\mathcal{F}(X)$ thanks to the universal linearization property of Lipschitz free spaces. Passing to the adjoint map of a linear projection from $\mathcal{F}(X)$ onto $\mathcal{F}(Y)$ we can see that it is also equivalent to the fact that there exists a weak$^*$-weak$^*$ continuous, linear and bounded operator $E\colon \text{Lip}_0(Y)\rightarrow\text{Lip}_0(X)$ that is an extension operator. 

On the other hand, property $(4)$ is equivalent to the existence of a (non necessarily weak$^*$-weak$^*$ continuous) linear and bounded extension operator from $\text{Lip}_0(Y)$ into $\text{Lip}_0(X)$ (we discuss this equivalence and further equivalent formulations of local complementation in section 5). 

With this in mind, it is straightforward to verify the following chain of implications: $(1)$ implies $(2)$, which is equivalent to $(3)$, which implies $(4)$. 

Lindenstrauss and Tzafriri proved in \cite{LinTza71} that if a Banach space $X$ is not a Hilbert space, then it contains a closed linear subspace which does not verify the Compact Extension Property (CEP) in $X$. On the other hand, Kalton showed in \cite{Kal84} that a subspace $Y$ of $X$ is locally complemented if and only if $Y$ has the CEP in $X$. Therefore, in every non-Hilbert Banach space we can find a linear subspace that fails the weakest, and thus all, of the projectional properties we have defined. Hence, it is natural to ask
if for given a Banach space $X$ and a closed linear subspace $Y$, there exists an intermediate subspace $Z$ of $X$ that contains $Y$, such that $Z$ has the same density character as $Y$ and $Z$ verifies one of the aforementioned projectional properties. In the linear projection case, these kinds of properties of the space $X$ are known as Complementation Properties.

More precisely and more generally, given $\alpha,\beta$ two cardinal numbers, we say that a Banach space $X$ has the \emph{($\alpha,\beta$) Complementation Property, CP$(\alpha,\beta)$} for short, if for every closed linear subspace $Y\subset X$ with $\text{dens}(Y)=\alpha$ there exists another subspace $Z$ that contains $Y$, such that $\text{dens}(Z)=\beta$ and $Z$ is linearly complemented in $X$. We say that $X$ has the \emph{Separable Complementation Property (SCP)} if it has the CP$(\aleph_0,\aleph_0)$. Analogously we define the following concepts:

\begin{itemize}
    \item[-] The ($\alpha,\beta$) Lipschitz Retraction Property (Lipschitz RP$(\alpha,\beta)$) and Separable Lipschitz Retraction Property (Lipschitz SRP).
    \item[-] The ($\alpha,\beta$) $\mathcal{F}$-Lipschitz Retraction Property ($\mathcal{F}$-Lipschitz RP$(\alpha,\beta)$) and Separable $\mathcal{F}$-Lipschitz Retraction Property ($\mathcal{F}$-Lipschitz SRP).\footnote{Although the Lipschitz RP$(\alpha,\beta)$ is equivalent to the $\mathcal{F}$-Lipschitz RP$(\alpha,\beta)$ in Banach spaces, it is not equivalent in the more general metric space setting which we will study in this paper, so we choose to define both concepts separately.}
    \item[-] The ($\alpha,\beta$) Local Complementation Property (Local CP$(\alpha,\beta)$) and Local Separable Complementation Property (Local SCP).
\end{itemize}

The CP($\alpha,\beta$), and especially the SCP, has been studied deeply, and there are both positive and negative results. For instance, every Banach space with the Plichko property has the SCP, while there are examples of Banach spaces with the Radon-Nikodym Property failing the SCP (we refer to the survey by Plichko and Yost \cite{PliYos00} for a detailed exposition of these properties). A remarkable result of Koszmider, Shelah and \'{S}wi\c{e}tek \cite{KosSheSwi18} implies that, under the Generalized Continuum Hypothesis, for every pair of cardinal numbers $\alpha\leq\beta$, there exists a connected compact space $K$ such that $C(K)$ fails the CP$(\alpha,\beta)$. 

The Local CP$(\alpha,\beta)$ situation is much simpler: It is a result of Heinrich and Mankiewicz (Proposition 3.4 in \cite{HeiMan82}) that every Banach space $X$ has the Local CP$(\alpha,\alpha)$ for any cardinal $\alpha$ smaller than the density character of $X$. 

The Lipschitz RP($\alpha,\beta$) introduced above is closely related to the other two properties. First of all, if a Banach space has the Lipschitz RP($\alpha,\beta$), then its Lipschitz free space has the CP($\alpha,\beta$). And secondly, a closed linear subspace $Y$ which is locally complemented in a Banach space $X$ is a Lipschitz retract if and only if it $Y$ is a Lipschitz retract of its bidual in its natural embedding; and it is a famous open question posed by Benyamini and Lindenstrauss if every separable Banach space is a Lipschitz retract of its bidual (see \cite{BenLin00}). Therefore, since every Banach space has the Local SCP, if this conjecture is true, then every Banach space has the Lipschitz SRP.   

We are interested in studying these properties in the more general setting of metric spaces. However, since metric spaces in general lack a linear structure, we cannot consider linear projections anymore. The three properties we study are the following:
\begin{definition}
Let $M$ be a metric space, and let $A$ be a closed subset of $M$. We say that:
\begin{itemize}
    \item[(A)] $A$ is a \emph{Lipschitz retract of $M$} if there exists a Lipschitz map $R\colon M\rightarrow A$ such that $R(a)=a$ for all $a\in A$. We call $R$ a \emph{Lipschitz retraction from $M$ onto $A$}.
    \item[(B)] $A$ is an \emph{$\mathcal{F}$-Lipschitz retract of $M$} if there exists a Lipschitz map $F\colon M\rightarrow \mathcal{F}(A)$ such that $F(a)=\delta(a)$ for all $a\in A$. We call $F$ an \emph{$\mathcal{F}$-Lipschitz retraction from $M$ onto $A$}. 
    \item[(C)] $A$ is {locally complemented in $M$} if there exists a linear and bounded extension operator from $\text{Lip}_0(A)$ into $\text{Lip}_0(M)$. 
\end{itemize}
We define the Lipschitz RP($\alpha,\beta$), $\mathcal{F}$-Lipschitz RP($\alpha,\beta$) and the Local CP($\alpha,\beta$) similarly as in the Banach space setting.
\end{definition}
Although we still have that $(A)$ implies $(B)$ which implies $(C)$, among metric spaces an $\mathcal{F}$-Lipschitz retract need not be a Lipschitz retract because a metric space is not always a Lipschitz retract of its Lipschitz free space (unlike a Banach space, thanks to the barycenter map), so the three properties are distinct. 

Sections 2,3 and 4 focus on the Lipschitz RP and the $\mathcal{F}$-Lipschitz RP. In section 2 we study positive results, finding sufficient conditions for a metric space to have the Lipschitz SRP and the $\mathcal{F}$-Lipschitz SRP. Specifically, we focus on the existence of Projectional skeletons in the Lipschitz free space of a metric space, which implies the $\mathcal{F}$-Lipschitz SRP. We study  the metric analogous concept of Lipschitz Retractional skeletons in metric spaces, which implies the Lipschitz SRP. We also give a geometric sufficient condition on a metric space $M$ for $\mathcal{F}(M)$ to be Plichko. 

Sections 3 and 4 are the most technical part of the paper. There, we give counterexamples for the Lipschitz RP. In section 3 we construct a metric space with the $\mathcal{F}$-Lipschitz RP such that any separable subset containing two specific points cannot be a Lipschitz retract, thus failing the Lipschitz SRP; while in section 4 we show the existence of a metric space failing the Lipschitz RP($\alpha,\alpha$) for any infinite cardinal $\alpha$. 

Finally, in section 5 we prove that, analogously to the Banach space setting, every metric space has the Local CP($\alpha,\alpha$) for any infinite cardinal $\alpha$. We do so by adapting an argument from Sims and Yost (\cite{SimYos89}) to metric spaces.

\section{Projectional skeletons in Lipschitz free spaces and Lipschitz retractional skeletons}

\subsection{Projectional skeletons in Lipschitz free spaces and the $\mathcal{F}$-Lipschitz SRP}
We start by recalling the concept of Projectional skeleton. Recall that a partially ordered set $\Gamma$ is \emph{directed} if for every $s_0,s_1\in\Gamma$ there is a $t\in\Gamma$ such that $s_0,s_1\leq t$. 
\begin{definition}[Kubi\'s \cite{Kub09}]
\label{skeletons}
Let $X$ be a Banach space. A \emph{projectional skeleton} on $X$ is a family $\{P_s\}_{s\in\Gamma}$ of bounded linear projections on $X$ indexed by a directed partially ordered set $\Gamma$, such that the following conditions hold:
\begin{itemize}
    \item[(i)] $P_sX$ is separable for all $s\in\Gamma$.
    \item[(ii)] $P_sP_t=P_tP_s=P_s$ whenever $s,t\in \Gamma$ and $s\leq t$.
    \item[(iii)] If $(s_n)_n$ is an increasing sequence of indices in $\Gamma$, then $s=\sup_{n\in\mathbb{N}}s_n$ exists and $P_sX=\overline{\bigcup_{n\in\mathbb{N}}P_{s_n}X}$.
    \item[(iv)] $X=\bigcup_{s\in\Gamma}P_sX$.
\end{itemize}

If $r\geq 1$, we say that an \emph{$r$-projectional skeleton} is a projectional skeleton where every projection has norm less or equal than $r$. 
We say that a projectional skeleton is commutative if $P_sP_t=P_tP_s$ for all $s,t\in\Gamma$, regardless whether they are comparable or not.
\end{definition}

It follows from the directedness of $\Gamma$ and properties $(iii)$ and $(iv)$ that if a Banach space $X$ admits a projectional skeleton, then it has the SCP.

Let us remark a few things about this definition. Firstly, it is straightforward to prove that condition $(ii)$ is equivalent to the fact that $P_s(X)\subset P_t(X)$ and $P_s^*(X^*)\subset P_t^*(X^*)$. 

Also, as proven by Kubi\'{s} in \cite{Kub09}, a projectional skeleton is always uniformly bounded in norm, so it is an $r$-projectional skeleton for some $r\geq 1$. This implies that condition $(iii)$ is equivalent to the fact that if $(s_n)_n$ is an increasing sequence of indices in $\Gamma$ and $s=\sup_{n\in\mathbb{N}}s_n$, then $\lim_{n\in\mathbb{N}}P_{s_n}(x)=P_s(x)$ in the norm topology.

We are interested in the case when $M$ is a complete metric space such that $\mathcal{F}(M)$ admits a projectional skeleton $\{P_s\}_{s\in\Gamma}$. As we said, this implies that $\mathcal{F}(M)$ has the SCP. In fact, thanks to the next result, it implies a slightly stronger version of the SCP for Lipschitz free spaces. Recall that $\Gamma'\subset\Gamma$ is $\sigma$-closed if for every increasing sequence $(s_n)_n\subset \Gamma$, the index $s=\sup_n (s_n)$ belongs to $\Gamma'$; and $\Gamma'$ is cofinal if for every $s\in \Gamma$ there exists $t\in \Gamma'$ such that $s\leq t$. If $\Gamma'\subset \Gamma$ is a $\sigma$-closed cofinal subset of $\Gamma$, then $\{P_s\}_{s\in\Gamma'}$ is also a projectional skeleton. For a detailed exposition and proof of these facts we refer to \cite{Kal20}.  

We use the concept of \emph{support} of an element of the Lipschitz free space defined by Aliaga, Perneck\'a, Petitjean and Proch\'azka in \cite{AliPerPetPro20}. Given an element $\mu$ of $\mathcal{F}(M)$, the \emph{support of $\mu$}, written $\text{supp}(\mu)$, is the intersection of all subsets $K$ in $M$ such that $\mu\in \mathcal{F}(K)$. The set $\text{supp}(\mu)$ is a closed separable subset of $M$ such that $\mu\in\mathcal{F}\big(\text{supp}(\mu)\big)$. 

\begin{proposition}
\label{reductiontofreeskeleton}
Let $M$ be a complete metric space (resp. Banach space), and let $\{P_s\}_{s\in\Gamma}$ be a projectional skeleton on $\mathcal{F}(M)$. Then there exists a $\sigma$-closed cofinal subset of $\Gamma$ and a family $\{A_s\}_{s\in\Gamma'}$ of separable subsets (resp. separable linear subspaces) of $M$  such that $P_s(\mathcal{F}(M))=\mathcal{F}(A_s)$. 

In particular, $\mathcal{F}(M)$ admits a projectional skeleton $\{P_s\}_{s\in\Gamma}$ such that $P_s(\mathcal{F}(M))$ is $\mathcal{F}(A_s)$ where $A_s$ is a separable subset (resp. separable linear subspace) of $M$ for all $s\in\Gamma$. 
\end{proposition}
\begin{proof}
Let $s=s_0\in\Gamma$ be an arbitrary index. The set $P_{s_0}(\mathcal{F}(M))$ is a separable subset of $\mathcal{F}(M)$. Since the support of a point in $\mathcal{F}(M)$ is a closed separable subset of $M$, we obtain that the set 
$$A_{s_0}=\overline{\bigcup_{\mu\in P_{s_0}(\mathcal{F}(M))}\text{supp}(\mu)} $$
is a closed separable subset of $M$ as well. Moreover, since $\text{supp}(\mu)\subset A_{s_0}$ for all $\mu\in P_{s_0}(\mathcal{F}(M))$, we have that $P_{s_0}(\mathcal{F}(M))\subset \mathcal{F}(A_{s_0})$, which is a separable subset of $\mathcal{F}(M)$. By the properties of projectional skeletons, we can find $s_1\in \Gamma$, with $s_0\leq s_1$ such that $\mathcal{F}(A_{s_0})\subset P_{s_1}(\mathcal{F}(M))$.

By induction, we construct $(s_n)_n\subset \Gamma$ such that $s_n\leq s_{n+1}$ and $\{A_{s_n}\}_n$ are closed separable subsets of $M$ verifying
$$ P_{s_n}\big(\mathcal{F}(M)\big)\subset \mathcal{F}(A_{s_n})\subset P_{s_{n+1}}\big(\mathcal{F}(M)\big). $$

This implies that $\bigcup_{n\in\mathbb{N}}P_{s_n}\big(\mathcal{F}(M)\big)=\bigcup_{n\in\mathbb{N}}\mathcal{F}(A_n)$. Consider now $t_s=\text{sup}(s_n)\in \Gamma$. We have then that 
$$P_{t_s}\big(\mathcal{F}(M)\big)=\overline{\bigcup_{n\in\mathbb{N}}P_{s_n}\big(\mathcal{F}(M)\big)}=\overline{\bigcup_{n\in\mathbb{N}}\mathcal{F}(A_n)}=\mathcal{F}\bigg(\overline{\bigcup_{n\in\mathbb{N}}A_n}\bigg). $$ 

Set $A_{t_s}=\overline{\bigcup_{n\in\mathbb{N}}A_n}$ and $\Gamma'=\{t_s\}_{s\in\Gamma}$. Clearly $\Gamma'$ is $\sigma$-complete and cofinal, and the result follows.

It is easy to modify this argument to see that if $X$ is a Banach space and $\mathcal{F}(X)$ admits a projectional skeleton $\{P_s\}_{s\in \Gamma}$, we can assume that $P_s\big(\mathcal{F}(X)\big)=\mathcal{F}(Y_s)$ where $\{Y_s\}$ is a family of separable linear subspaces.
\end{proof}

Therefore, applying again the directedness of $\Gamma$ and properties $(iii)$ and $(iv)$ of projectional skeletons, we obtain the following corollary:

\begin{corollary}
Let $M$ be a complete metric space. If $\mathcal{F}(M)$ admits a projectional skeleton then $M$ has the $\mathcal{F}$-Lipschitz SRP.
\end{corollary}
\begin{proof}
Let $\{P_s\}_{s\in\Gamma}$ be a projectional skeleton such that $P_s\big(\mathcal{F}(M)\big)$ is $\mathcal{F}(A_s)$ where $A_s$ is a separable subset of $M$ for all $s\in\Gamma$, and let $N\subset M$ be a separable subset of $M$. Since $\mathcal{F}(N)$ is separable, there exists a sequence $(s_n)_n$ in $\Gamma$ such that $\mathcal{F}(N)\subset \overline{\bigcup_{n\in\mathbb{N}}\mathcal{F}(A_{s_n})}$. By the directedness of $\Gamma$, we may assume that $(s_n)_n$ is increasing. Hence, the supremum $s_0=\sup_{n\in\mathbb{N}}s_n$ exists and $\mathcal{F}(A_{s_0})=\overline{\bigcup_{n\in\mathbb{N}}\mathcal{F}(A_{s_n})}$. Since $\mathcal{F}(N)\subset\mathcal{F}(A_{s_0})$, we have that $N\subset A_{s_0}$. Finally, since $\mathcal{F}(A_{s_0})$ is complemented in $\mathcal{F}(M)$, we obtain that $A_{s_0}$ is an $\mathcal{F}$-Lipschitz retraction of $M$, which finishes the proof.
\end{proof}

Notice that if $\mathcal{F}(M)$ only has the usual SCP, given a separable subset $A$ of $M$, we have that there exists a separable subspace $Z$ of $\mathcal{F}(M)$ such that $Z$ is linearly complemented in $\mathcal{F}(M)$. However, $Z$ need not be the Lipschitz free space of a separable subset of $M$, so we cannot deduce that $M$ has the $\mathcal{F}$-Lipschitz SRP directly in this way just from the assumption that $\mathcal{F}(M)$ has the SCP. 

A Banach space $X$ is \emph{$\lambda$-Plichko} if there exists a $\lambda$-norming subspace $N\subset X^*$ and a linearly dense subset $\Delta\subset X$ such that for every $f\in N$, the set $\{x\in \Delta\colon \langle f,x\rangle\neq 0\}$ is countable. It is a result of Kubi\'{s} (\cite{Kub09}) that a Banach space is $\lambda$-Plichko if and only if it admits a commutative $\lambda$-projectional skeleton. We are going to study a geometric condition on metric spaces that implies that $\mathcal{F}(M)$ is $1$-Plichko. This result has been obtained by the second author in collaboration with A.J. Guirao and V. Montesinos, and will be included in a future note. We include the full proof here for completeness.

\begin{theorem}
\label{Plichkoinfreespaces}
Let $M$ be a complete metric space and $\lambda\geq 1$. If there exists a point $p\in M$ such that for every $x\in M$ and every $r<\frac{1}{\lambda}$ the ball $B\big(x,r\cdot d(x,p)\big)$ is separable, then $\mathcal{F}(M)$ is $\lambda$-Plichko. In particular, $M$ has the $\mathcal{F}$-Lipschitz SRP.
\end{theorem}

In order to prove this theorem, we need two auxiliary results. First, define 
$$S_0=\{f\in\text{Lip}_0(M)\colon \text{supp}(f)\text{ is separable }\}\subset \text{Lip}_0(M).$$
This set is clearly a closed linear subspace of $\text{Lip}_0(M)$. Moreover, we have the following proposition:

\begin{proposition}
\label{separableisnorming}
Let $M$ be a metric space and $\lambda\geq 1$. If for every $p\in M$, and every $0<r<\frac{1}{\lambda}$, the set $B(p,r\cdot d(p,0))$ is countable, then $S_0$ is $\lambda$-norming.
\end{proposition}
\begin{proof}
By Lemma 3.3 of \cite{Kal04}, it is enough to show that for every finite set $F\subset M$ with $0\in F$, every $\varepsilon>0$ and every Lipschitz function $f\in \text{Lip}_0(F)$ with $\|f\|_\text{Lip}=1$ there exists a function $g\in S_0$ such that $g_{|F}=f$ and $\|g\|_\text{Lip}\leq \lambda(1+\varepsilon)$. Using McShane's extension theorem, this is equivalent to proving that for every finite set $F\subset M$ with $0\in F$, every $\varepsilon>0$ and every function $f\in \text{Lip}_0(M)$ with $\|f\|_\text{Lip}=1$, there exists a function $g\in S_0$ such that $g_{|F}=f_{|F}$ and $\|g\|_\text{Lip}\leq \lambda(1+\varepsilon)$.

Fix $f\in \text{Lip}_0(M)$ with $\|f\|_\text{Lip}=1$. Define the subsets $P=\{p\in M:f(p)>0\}$, $N=\{p\in M:f(p)<0\}$ and $Z=\{p\in M:f(p)=0\}$. 

Fix $x_0\in P$ and $\varepsilon>0$, and define $\tau_{x_0}(p)=\max\{f(x_0)-\lambda(1+\varepsilon)d(p,x_0),0\}$. Put $D_{x_0}=\{p\in M,~ \tau_{x_0}(p)>0\}$ ($D_{x_0}$ is the topological interior of the support of $\tau_{x_0}$). We claim that $D_{x_0}\subset P$. Indeed, let $p\in D_{x_0}$. Then $\tau_{x_0}(p)=f(x_0)-\lambda(1+\varepsilon)d(p,x_0)>0$. Equivalently, $d(p,x_0)<\big(\lambda(1+\varepsilon)\big)^{-1}f(x_0)$. 

Also, since $\|f\|_\text{Lip}=1$, we have that $|f(x_0)-f(p)|<\big(\lambda(1+\varepsilon)\big)^{-1}f(x_0)$. Thus,

$$f(p)\geq f(x_0)-(\lambda(1+\varepsilon))^{-1}f(x_0)=f(x_0)\Big(1-(\lambda(1+\varepsilon))^{-1}\Big)>0, $$
as we claimed. It is also clear that $x_0\in D_{x_0}$. It follows that $P=\bigcup_{x\in P} D_{x}$.

Similarly, for $x_0\in N$ and $\varepsilon>0$, we define $\tau_{x_0}=\min\{f(x_0)+\lambda(1+\varepsilon)d(p,x_0),0\}$ and $D_{x_0}=\{p\in M,~ \tau_{x_0}(p)<0\}$. Following the same reasoning as before, we get $N=\bigcup_{x\in N}D_x$. In particular, if $x\in P$ and $y\in N$, we get $D_x\cap D_y=\emptyset$. 

Now let $F\subset M$ be a finite set with $0\in F$. Put $FP=F\cap P$, $FN=F\cap N$ and $FZ=F\cap Z$. Define a function $g\colon M\rightarrow M$ in the following way:
$$ 
g(p)=
\begin{cases}
\bigvee\limits_{x\in FP} \tau_{x}(p), &\text{ if }p\in P\\
\bigwedge\limits_{x\in FN} \tau_{x}(p), &\text{ if }p\in N\\
0, &\text{ if }p\in Z
\end{cases}
$$
This function has the desired properties, that is:

\begin{itemize}
    \item[(i)] $g(p)=f(p)$ for all $p\in F$,
    \item[(ii)] $g(0)=0$,
    \item[(iii)] $g\in S_0$, and
    \item[(iv)] $\|g\|_\text{Lip}\leq \lambda(1+\varepsilon)$.
\end{itemize}

Let us check this. Let $p\in F$. Suppose that $p\in FP$. Then $g(p)\geq \tau_p(p)=f(p)$ by definition. Let $x$ be an arbitrary point in $FP$. Then, since $\|f\|_\text{Lip}=1$ and $\lambda\geq 1$:

\begin{align*}
    \tau_{x}(p)&=f(x)-\lambda(1+\varepsilon)d(p,x)=f(x)-\lambda d(p,x)-\lambda\varepsilon d(p,x)\\
    &\leq f(x)-d(p,x)\leq f(x)-(f(x)-f(p))= f(p)
\end{align*}
Hence $g(p)=f(p)$. By a similar argument we see that if $q\in FN$, then $g(q)=f(q)$, and clearly if $z\in FZ$, by definition $g(z)=f(z)=0$. We have proven $(i)$ and $(ii)$ since $0\in F$.

To see $(iii)$, we need to prove that $g$ has a separable support. Note that $\text{supp}(g)=\bigcup_{x\in F} \text{supp}(\tau_x)$. Since $F$ is finite, it suffices to show that $\text{supp}(\tau_x)$ is separable for every $x\in F$. Suppose $x_0\in FP$ and let $p\in M$ with $d(p,x_0)>(\lambda(1+\varepsilon))^{-1}d(x_0,0)$. Then $\lambda(1+\varepsilon )d(p,x_0)>d(x_0,0)$, so

$$f(x_0)-\lambda(1+\varepsilon)d(p,x_0)>f(x_0)-d(x_0,0)<f(0)=0,$$
which implies that $\tau_{x_0}(p)=0$. Thus $\text{supp}(\tau_{x_0})\subset B(x_0, (\lambda(1+\varepsilon))^{-1}d(x_0,0))$, which is separable by hypothesis. The same reasoning applies if $x_0\in FN$, so we conclude that $g$ has separable support and thus condition $(iii)$ is verified.

Property $(iv)$ follows from the definition of $\tau_x$ for every $x\in M$. 
\end{proof}

We have now a norming subspace of $\text{Lip}_0(M)$, and we are going to find a linearly dense subset $\Delta$ in $\mathcal{F}(M)$ such that $S_0$ is countably supported in $\Delta$. Notice that if a subset $D$ of $M$ is dense, the corresponding subset $\Delta_D=\{\delta(p)\colon p\in D\}$ in $\mathcal{F}(M)$ is linearly dense. Hence, if we find a dense subset of $M$ such that its intersection with each separable subset of $M$ is countable, we will be able to prove Theorem \ref{Plichkoinfreespaces}. Such a dense subset does not exist for every metric space, but fortunately, the same geometric condition on the separability of the balls around every point except $0$ we used to prove $S_0$ is norming is sufficient to construct a dense set with this property. This follows from a standard maximality argument, but we include the proof for completeness.
\begin{lemma}
\label{Densesubsetwithcountableintersection}
Let $M$ be a metric space such that every point in $M$ has a separable neighborhood ($M$ is locally separable). Then there exists a dense set $D$ in $M$ such that for every separable subset $S$ of $M$, the intersection $D\cap S$ is countable. 
\end{lemma}
\begin{proof}
Consider the following set:

\begin{align*}
    T=\big\{\{A_i\}_{i\in I}\subset\mathcal{P}(M)\colon &A_i\text{ is non-empty, open and separable for all }i\in I,\\ &A_i\cap A_j=\emptyset\text{ for all }i\neq j\in I\big\}, 
\end{align*}
which is non-empty since $M$ is locally separable. The set $T$ can be ordered by inclusion, and it is straightforward to check that every chain in $T$ has an upper bound given by the union of every family in the chain. Hence, by Zorn's Lemma we can consider $F_0=\{A_i\}_{i\in I}$ a maximal family in $T$. Then, since $F_0$ is maximal in $T$ and $M$ is locally separable, we have that $M=\overline{\bigcup_{i\in I}A_i}$.

Choose for every $i\in I$ a countable set $D_i$ dense in $A_i$, and set $D=\bigcup_{i\in I}D_i$. Let us check that $D$ verifies the thesis of the Lemma: Let $S$ be a separable subset of $M$. Then $S$ has the countable chain condition, so there exists a countable subset $F_0'=\{A_{i_n}\}_{n\in\mathbb{N}}$ of $F_0$ such that 

$$S\cap \bigcup_{i\in I}A_i=S\cap \bigcup_{n\in\mathbb{N}}A_{i_n}.$$
Therefore, since $D_i\subset A_i$ for all $i\in I$, we obtain that $S\cap D=\bigcup_{n\in\mathbb{N}}S\cap D_{i_n}$, which is countable since it is the countable union of countable sets. 

\end{proof}

Finally, we can prove Theorem \ref{Plichkoinfreespaces}:

\begin{proof}[Proof of Theorem \ref{Plichkoinfreespaces}]
We may assume without loss of generality that $p=0$ is the distinguished point of $M$, since the Lipschitz free spaces of the same metric space with different distinguished points are linearly isometric. 

Put $N=S_0$ as a closed subspace of $\text{Lip}_0(M)$. By proposition \ref{separableisnorming}, $N$ is $\lambda$-norming. By hypothesis, the set $M\setminus \{0\}$ is locally separable, so by Lemma \ref{Densesubsetwithcountableintersection}, we can find $D'\subset M\setminus\{0\}$ dense such that $D'$ intersects every separable subset of $M\setminus\{0\}$ in a countable set. Clearly, the set $D=D'\{0\}$ also verifies that it is dense in $M$ and for every separable subset $S$ of $M$, the intersection $D\cap S$ is countable. Put $\Delta=\{\delta(x)\colon x\in D\}$. Then $\Delta$ is linearly dense in $\mathcal{F}(M)$, and for every $f\in S_0$ we have that 
$$\{x\in D\colon \langle f,\delta(x)\rangle\neq 0\}= \text{supp}(f)\cap D$$
is countable. We conclude that $\mathcal{F}(M)$ is $\lambda$-Plichko.
\end{proof}

\subsection{Lipschitz retractional skeletons in metric spaces and the Lipschitz SRP}
We pass now to  studying the Lipschitz SRP. In \cite{Kal20}, Kalenda studies the concept of retractional skeleton in the context of compact Hausdorff spaces, as an analogous concept to projectional skeletons in the topological setting. In metric spaces, we can define \emph{Lipschitz retractional skeletons}:
\begin{definition}
Let $M$ be a metric space. A \emph{Lipschitz retractional skeleton} is a set $\{R_s\}_{s\in\Gamma}$ of Lipschitz retractions in $M$ indexed by a directed partially ordered $\sigma$-complete set $\Gamma$, such that the following conditions hold:
\begin{itemize}
    \item[(i)] $R_s(M)$ is separable for all $s\in \Gamma$. 
    \item[(ii)] $R_sR_t=R_tR_s=R_s$ whenever $s,t\in\Gamma$ and $s\leq t$.
    \item[(iii)] If $(s_n)_n$ is an increasing sequence of indices in $\Gamma$, then $s=\sup_{n\in\mathbb{N}}s_n$ exists in $\Gamma$ and $R_s(M)=\overline{\bigcup_{s_\in\mathbb{N}}R_{s_n}(M)}$.
    \item[(iv)] $M=\bigcup_{s\in\Gamma}R_s(M)$.
\end{itemize}
If $\|R_s\|_\text{Lip}\leq r$ for all $s\in\Gamma$, then we say that $\{R_s\}_{s\in\Gamma}$ is an $r$-Lipschitz retractional skeleton. If $R_sR_t=R_tR_s$ for all $s,t\in \Gamma$, then we say that the Lipschitz retractional skeleton is commutative. 
\end{definition}

As in the linear case, it is straightforward to see that if $M$ admits a Lipschitz retractional skeleton, then $M$ has the Lipschitz SRP.

Thanks to the linearization property of Lipschitz free spaces, we can deduce the existence of a projectional skeleton in $\mathcal{F}(M)$ provided $M$ admits a Lipschitz retractional skeleton:

\begin{proposition} 
\label{metricprojectiveskeelton}
Let $M$ be a complete metric space and $r\geq 1$. Suppose that $M$ admits a (commutative) $r$-Lipschitz retractional skeleton on $M$.  Then $\mathcal{F}(M)$ admits a (commutative) $r$-projectional skeleton.
\end{proposition}
\begin{proof}
Let $\{R_s\}_{s\in\Gamma}$ be an $r$-Lipschitz retractional skeleton in $M$. Let $P_s\colon =\widehat{R_s}\colon \mathcal{F}(M)\rightarrow\mathcal{F}(M)$ be the linear maps such that $\|P_s\|=\|R_s\|_\text{Lip}$ and $P_s(\delta(x))=\delta(R_s(x))$ for all $x\in M$. Let us check that this family is a projectional skeleton on $\mathcal{F}(M)$. 

In the first place, since $R_s(M)$ is separable for all $s\in \Gamma$, and $P_s(\mathcal{F}(M))$ is equal to  $\mathcal{F}(R_s(M))$, we obtain that $P_s(\mathcal{F}(M))$ is separable for all $s\in \Gamma$. Next, suppose $s,t\in \Gamma$ with $s\leq t$ and take $x\in M$. We have then that 
\begin{align*}
    P_sP_t(\delta(x))=\delta(R_sR_t(x))=\delta(R_s(x))=P_s(\delta(x)),
\end{align*}
and similarly for $P_tP_s(\delta(x))$. Since $P_sP_t$, $P_tP_s$ and $P_s$ are bounded linear maps and $\delta(M)$ is a linearly dense subset of $\mathcal{F}(M)$, we obtain that $P_sP_t=P_tP_s=P_s$ as desired. 

Next, suppose that $(s_n)_n$ is an increasing sequence of indices in $\Gamma$, and let $s=\sup_{n\in\mathbb{N}}s_n$. Consider $x\in M$ and $\varepsilon>0$. By hypothesis there exists $n_0\in\mathbb{N}$ and $y\in M$ such that $d\big(R_s(x),R_{s_{n_0}}(y)\big)<\varepsilon$. Hence, since the $\delta$ map is an isometry, we have that $\|\delta\big(R_s(x)\big)-\delta\big(R_{n_0}(y)\big)\|<\varepsilon$. This implies that $\|P_s(\delta(x))-P_{n_0}(\delta(y))\|<\varepsilon$. Hence $P_s(\delta(x))\in\overline{\bigcup_{n_\in\mathbb{N}}P_{s_n}(\mathcal{F}(M))}$. 

Now, since $(s_n)$ is increasing, by the remark we made about condition $(ii)$ of Definition \ref{skeletons}, the family $\{P_{s_n}(\mathcal{F}(M))\}_n$ is increasing as well. This implies that $\overline{\bigcup_{n_\in\mathbb{N}}P_{s_n}(\mathcal{F}(M))}$ is a linear subspace of $\mathcal{F}(M)$. Then, by the linearity of $P_s$ and the fact that $\delta(M)$ is linearly dense in $\mathcal{F}(M)$, we obtain that

$$P_s(\mathcal{F}(M))=\overline{\bigcup_{n_\in\mathbb{N}}P_{s_n}(\mathcal{F}(M))},$$
as desired.

Finally, to prove that $\mathcal{F}(M)=\bigcup_{s\in\Gamma}P_s(\mathcal{F}(M))$, we again use the concept of support of an element of $\mathcal{F}(M)$. For all $\mu\in\mathcal{F}(M)$ the set $\text{supp}(\mu)\subset M$ is a closed separable subset such that $\mu\in\mathcal{F}(\text{supp}(\mu))$. Hence, if for every $\mu\in\mathcal{F}(M)$ we find $s\in\Gamma$ such that $\text{supp}(\mu)\subset R_s(M)$, we will obtain that $\mu\in\mathcal{F}(R_s(M))=P_s(\mathcal{F}(M))$, completing the proof.

To this end, consider $\mu\in\mathcal{F}(M)$, and let $(x_n)_n\subset \text{supp}(\mu)$ be a dense sequence. By hypothesis, for $x_1$ we can find $s_1\in\Gamma$ such that $x_1\in R_{s_1}(M)$. Suppose that we have constructed $(s_i)_{i=1}^n$ in $\Gamma$ such that $s_i\leq s_{i+1}$ for $1\leq i\leq n-1$ and such that $x_i\in R_{s_i}(M)$. By hypothesis there exists $s^*\in \Gamma$ such that $x_{n+1}\in R_{s^*}(M)$. Since $\Gamma$ is directed, we can find $s_{n+1}\in \Gamma$ such that $s_i\leq s_{n+1}$ for $1\leq i\leq n$ and $s^*\leq s_{n+1}$. 

This way we inductively construct an increasing sequence $(s_n)_n$ such that $x_n\in R_{s_n}(M)$. By item $(iii)$ in the hypothesis, there exists $s\in \Gamma$ such that $R_s(M)=\overline{\bigcup_{n_\in\mathbb{N}}R_{s_n}(M)}$. Since the dense sequence $(x_n)_n$ is contained in $\bigcup_{n_\in\mathbb{N}}R_{s_n}(M)$, it follows that $\text{supp}(\mu)\subset R_s(M)$. We conclude that $\{P_s\}_{s\in\Gamma}$ is a projective skeleton. The last two statements follow immediately.
\end{proof}

This yields the following result about Lipschitz free spaces of $1$-Plichko spaces.

\begin{corollary}
\label{XplichkoF(X)plichko}
Let $X$ be a $1$-Plichko Banach space. Then $\mathcal{F}(X)$ is $1$-Plichko.
\end{corollary}
\begin{proof}
If $X$ is $1$-Plichko, then it admits a commutative $1$-projectional skeleton. Since linear projections are in particular Lipschitz retractions, by Proposition \ref{metricprojectiveskeelton} the space $\mathcal{F}(X)$ also admits a commutative $1$-projectional skeleton. Hence $\mathcal{F}(X)$ is $1$-Plichko.
\end{proof}

\begin{remark}
In Proposition \ref{metricprojectiveskeelton}, we use the linearization property of Lipschitz free spaces to obtain projections in $\mathcal{F}(M)$ onto $\mathcal{F}(A_s)$ from a Lipschitz retraction in $M$ onto $A_s$. In general, if $\mathcal{F}(A_s)$ is complemented in $\mathcal{F}(M)$, we do not necessarily have that $A_s$ is a Lipschitz retraction of $M$. It is straightforward to prove that if $A_s$ if a Lipschitz retraction of $\mathcal{F}(A_s)$ (for instance, if $A_s$ is a Banach space, thanks to the barycenter map), then the two statements are indeed equivalent. 

However, this is not enough to prove the converse of Proposition \ref{metricprojectiveskeelton} even for Banach spaces, since the commutativity of the retractions (property (ii) in the definition of Lipschitz retractional skeleton) is lost in the process. Indeed, the converse is not true for general metric spaces: we construct in the next section a metric space without the Lipschitz SRP (and therefore without a Lipschitz retractional skeleton), but whose Lipschitz free space admits a projectional skeleton.
\end{remark}

To end the discussion about positive results in the Lipschitz SRP, let us comment that $C(K)$ Banach spaces for any compact Hausdorff space $K$ have the Lipschitz SRP:

\begin{proposition}
The Banach space $C(K)$ of real continuous functions has the Lipschitz SRP for any compact Hausdorff space $K$.
\end{proposition}
\begin{proof}
Let $Y$ be a separable linear subspace of $C(K)$. Then, there exists a separable linear subspace of $C(K)$ that contains $Y$ and is isometric to a $C(K')$ space for some compact metric space $K'$ (see Exercise 5.88 in \cite{FabHabHajMonZiz11}). By Theorem 3.5 in \cite{Kal07}, $A$ is an absolute $2$-Lipschitz retract, so in particular it is $2$-Lipschitz retract of $C(K)$, which concludes the proof. 

\end{proof}

This shows that the Lipschitz free space of any $C(K)$ space has the SCP, in contrast to many $C(K)$ spaces themselves, like $\ell_\infty$, which strongly fail this property. Moreover, as we mentioned in the introduction, it was proven in \cite{KosSheSwi18} that under the Generalized Continuum Hypothesis, for any cardinality $\alpha$ we can find a compact Hausdorff space $K$ of such that $C(K)$ has density character $\alpha$, and it does not have any nontrivial complemented linear subspace.  

\section{Metric space with almost no separable Lipschitz retracts}
In this section we are going to construct a complete metric space $M$ such that no separable subspace containing two specific points is a Lipschitz retract of $M$. Hence, this metric space strongly fails the Lipschitz SRP, and does not admit a Lipschitz retractional skeleton. However, as we remark at the end of the section, its Lipschitz free space does admit a commutative $1$-projectional skeleton, and it is thus $1$-Plichko. 

\subsection{Fat subsets of $[0,1]$}
We are going to define certain nowhere dense, compact subsets of $[0,1]$ with positive measure, each of them associated to a particular decreasing sequence of real numbers. For the rest of this section, fix $0<\varepsilon_0<1/2$ and let $\mathbb{Q}\cap [0,1]=(q_n)_{n=1}^\infty$ be a fixed ordering of the rational numbers in the unit interval. 

Consider a decreasing sequence of real numbers $\gamma=(\gamma_i)_{i=1}^\infty$ such that 

\begin{itemize}
    \item[(i)] $\gamma_i>0$ for all $i\in \mathbb{N}$,
    \item[(ii)] $\sum_{i=1}^\infty \gamma_i\leq 1-\varepsilon_0$,
    \item[(iii)] $q_1+\gamma_1<1$.
\end{itemize}

Put $\Gamma=\{\gamma=(\gamma_i)_i\colon \gamma\text{ is decreasing and verifies (i), (ii) and (iii)}\}$ for the rest of the section.

For any given $\gamma\in \Gamma$, we define inductively the following intervals:
\begin{align*}
    C^\gamma_1&=(q_1,q_1+\gamma_1)\subset (0,1),\\
    C^\gamma_{i}&=(q_{n_i},q_{n_i}+\gamma_i),\text{ where }n_i=\min\bigg\{n\in\mathbb{N}\colon (q_n,q_n+\gamma_i)\subset (0,1)\setminus\bigg(\bigcup_{j<i} C^\gamma_j\bigg) \bigg\}. 
\end{align*}
Note that $n_i$ as defined might not exist for some $i\in \mathbb{N}$. In that case we simply put $C^\gamma_{i}=\emptyset$ and go on to the next index. By property $(ii)$, there are infinitely many $i\in\mathbb{N}$ such that $C^\gamma_i$ is nonempty. 

Using this, we define the closed subset $L_\gamma\subset [0,1]$ as

$$L_\gamma=[0,1]\setminus \bigg(\bigcup_{i=1}^\infty C^\gamma_i\bigg) $$

We call the nonempty $C^\gamma_i=(q_{n_i},q_{n_i}+\gamma_i)$ sets the \textit{gaps of $L_\gamma$}, and we refer to the points $q_{n_i}, q_{n_i}+\gamma_i$ as \textit{endpoints of $C^\gamma_i$ in $L_\gamma$}. 

\begin{proposition}
\label{propertieslgamma}
Let $0<\varepsilon_0<1/2$, $\gamma=(\gamma_i)_{i=1}^\infty \in \Gamma$, and let $C^\gamma_i$ and $L_\gamma$ be defined as above. Then $L_\gamma$ is a compact subset of $[0,1]$ that verifies:

\begin{itemize}
    \item[(1)] The Lebesgue measure of $L_\gamma$, denoted $\mu(L_\gamma)$, is greater than or equal to $\varepsilon_0$.
    \item[(2)] The points $0$ and $1$ belong to $L_{\gamma}$ for every $\gamma\in \Gamma$. 
    \item[(3)] For any gap $C^\gamma_i=(q_{n_i},q_{n_i}+\gamma_i)$, the endpoints $q_{n_i},q_{n_i}+\gamma_i$ belong to $L_\gamma$.  
    \item[(4)] If $x,y\in L_\gamma$ with $x<y$ and $(x,y)\cap L_\gamma=\emptyset$, then there exists a $k\in \mathbb{N}$ such that $x=q_{n_k}$ and $y=q_{n_k}+\gamma_k$; that is, $(x,y)$ is a gap of $L_\gamma$ and $x$ and $y$ are its endpoints in $L_\gamma$.
    \item[(5)] The set $L_\gamma$ does not contain any nontrivial interval. Consequently, it is nowhere dense, and if $(x,x+\delta)\cap L_\gamma\neq \emptyset$ for some $\delta>0$, then $(x,x+\delta)$ contains an infinite set of endpoints of $L_\gamma$.
    
\end{itemize}
\end{proposition}
\begin{proof}

Notice that the Lebesgue measure of $L_\gamma$ is greater than or equal to $1-\sum_{i=1}^\infty \gamma_i\geq\varepsilon_0$ for all possible $\gamma$, and the points $0$ and $1$ are always in $L_\gamma$, so $(1)$ and $(2)$ are clear. Also, the gaps $(q_{n_i},q_{n_i}+\gamma_i)$ and $(q_{n_j},q_{n_j}+\gamma_j)$ are disjoint for different $i,j\in\mathbb{N}$, from which $(3)$ follows as well. Moreover, this also implies that if $x,y\in L_\gamma$ with $x<y$ and $(x,y)\cap L_\gamma=\emptyset$, then there must exist $k\in \mathbb{N}$ such that $x=q_{n_k}$ and $y=q_{n_k}+\gamma_k$.  That is, $x$ and $y$ are the endpoints of $L_\gamma$, and we have $(4)$.

Finally, the set $L_\gamma$ is nowhere dense, since it contains no intervals. Indeed, suppose there is an interval $(x,x+\delta)\subset L_\gamma$ for some $\delta>0$ with $x+\delta<1$. The subinterval $(x,x+\delta/2)$ contains a rational number $q_{n_0}$. Since $(\gamma_i)_{i=1}^\infty$ is decreasing and converging to $0$, there must exist $i_0$ such that $\gamma_{i}<\delta/2$ for all $i\geq i_0$. Then, for all $i\geq i_0$, the natural number $n_0$ verifies that $(q_{n_0},q_{n_0}+\gamma_i)\subset L_\gamma$, and in particular
$$(q_{n_0},q_{n_0}+\gamma_i)\subset (0,1)\setminus\bigg(\bigcup_{j<i_0}C_j^\gamma\bigg). $$
Therefore, there must exist $i_1\geq i_0$ such that $n_0= \min\bigg\{n\in\mathbb{N}\colon (q_n,q_n+\gamma_{i_1})\subset(0,1)\setminus\bigg(\bigcup_{j<i}C_j^\gamma\bigg)\bigg\}$, which implies that $C^\gamma_{i_1}=(q_{n_0},q_{n_0}+\gamma_{i_1})$, a contradiction with the assumption that $(q_{n_0},q_{n_0}+\delta/2)\subset L_\gamma$. The last statement of (5) follows from what we just proved and property (4).

\end{proof}
\subsection{Lipschitz functions from $L_\gamma$ to $L_\xi$.}
We can see that $\Gamma$ is an uncountable set. Indeed, given countably many sequences verifying the three properties listed above, by a diagonal argument it is easy to construct a different sequence that still verifies these properties and is different from all given sequences. In fact, we are going to prove the following stronger result, which is the fundamental property of the sets $L_\gamma$ that will be used to construct the metric space without Lipschitz SRP:
\begin{theorem}
\label{nolipschitzmap}
Let $(\gamma^n)_n\subset \Gamma$ be a countable family of sequences in $\Gamma$ and let $K\geq 1$. Then, there exists a $\gamma^*\in \Gamma$ such that there is no $K$-Lipschitz function $F\colon {L_{\gamma^*}}\rightarrow L_{\gamma^n}$ with $F(0)=0$ and $F(1)=1$. 
\end{theorem}

The next elementary proposition allows us to assume without loss of generality that the Lipschitz maps we consider are non-decreasing.
\begin{proposition}
\label{wlogFisnondecreasing}
Let $A,B\subset [0,1]$ be nonempty subsets with $0,1\in A\cap B$, and let $F\colon A\rightarrow B$ be a Lipschitz function such that $F(0)=0$ and $F(1)=1$. Then there exists a non-decreasing Lipschitz function $\widehat{F}\colon A\rightarrow B$ with $\|\widehat{F}\|_\text{Lip}\leq\|F\|_\text{Lip}$ such that $\widehat{F}(0)=0$ and $\widehat{F}(1)=1$.
\end{proposition}
\begin{proof}
Define $\widehat{F}\colon A\rightarrow B$ by
$$ \widehat{F}(x)=\max_{y\leq x} F(y).$$

Clearly $\widehat{F}$ is non-decreasing with $F\leq \widehat{F}$, $\widehat{F}(0)=0$, and $\widehat{F}(1)=1$. Given $p,q\in A$ with $q\leq p$, we have that $\widehat{F}(p)=F(z)$ for some $z\leq p$. If $z\leq q$ we necessarily have that $\widehat{F}(p)=\widehat{F}(q)$. Otherwise, we obtain:

$$\frac{\widehat{F}(p)-\widehat{F}(q)}{p-q}=\frac{F(z)-\widehat{F}(q)}{p-q}\leq \frac{F(z)-F(q)}{z-q}\leq \|F\|_\text{Lip},$$
which implies that $\|\widehat{F}\|_\text{Lip}\leq \|F\|_\text{Lip}$.
\end{proof}

Let us first give some definitions and prove some technical results which make the construction of $\gamma^*$ simpler: Let $\gamma,\xi\in \Gamma$ be two different sequences, and suppose there is a Lipschitz function $F\colon {L_{\gamma}}\rightarrow L_{\xi}$ such that $F(0)=0$ and $F(1)=1$. We say that a gap $C^\gamma_i=(q_{n_i},q_{n_i}+\gamma_i)$ \textit{jumps over a gap $C^\xi_j=(q_{n_j},q_{n_j}+\xi_j)$ with respect to $F$} if $F(q_{n_i})<q_{n_j}$ and $F(q_{n_i}+\gamma_i)>q_{n_j}+\xi_j$ (see Figure \ref{examplegapjump}).

The first lemma we prove says intuitively that if we have a monotonically nondecreasing Lipschitz function $F$ from $L_\gamma$ to $L_\xi$ that fixes $0$ and $1$, then every gap in $L_\xi$ must be jumped by a gap in $L_\gamma$ with respect to $F$. Although this result is fairly intuitive, we include the (simple) proof for completeness.

\begin{lemma}
\label{gapjumpsovergap}
Let $\gamma,\xi\in \Gamma$, with $L_\gamma$ and $L_\xi$ being the corresponding subsets. Suppose that there is a non-decreasing Lipschitz function $F\colon {L_{\gamma}}\rightarrow L_{\xi}$ such that $F(0)=0$ and $F(1)=1$. Let $C^\xi_j$ be a gap in $L_\xi$. Then there exists a gap in $L_\gamma$ that jumps over $C^\xi_j$ with respect to $F$.
\end{lemma}
\begin{proof}
Let $C^\xi_j=(q_{n_j},q_{n_j}+\xi_j)$. Consider the points:
\begin{align*}
    p_-&=\max\{F(x)\in L_\xi\colon~x\in L_\gamma,~F(x)\leq q_{n_j}\},\\
    p_+&=\min\{F(y)\in L_\xi\colon~y\in L_\gamma,~F(y)\geq q_{n_j}+\xi_{n_j}\}. 
\end{align*}
These minimum and maximum values always exist since we have that $F(0)=0$ and $F(1)=1$ and $L_\gamma$ is compact. Hence, we can find 
\begin{align*}
    x_-&=\max\{x\in L_\gamma\colon~F(x)=p_-\},\\
    y_+&=\min\{y\in L_\gamma\colon~F(y)=p_+\}.
\end{align*}

Since $F$ is non-decreasing and $p_-<p_+$, we have that $x_-<y_+$ and $(x_-,y_+)\cap L_\gamma=\emptyset$, so by Proposition \ref{propertieslgamma} $(3)$ there exists a $k\in \mathbb{N}$ such that $x_-=q_{n_k}$ and $y_+=q_{n_k}+\gamma_{k}$. The gap $C^\gamma_k=(q_{n_k},q_{n_k}+\gamma_k)$ jumps over $C^\xi_j$ with respect to $F$.
\end{proof}
\tikzset{every picture/.style={line width=0.75pt}} %set default line width to 0.75pt        
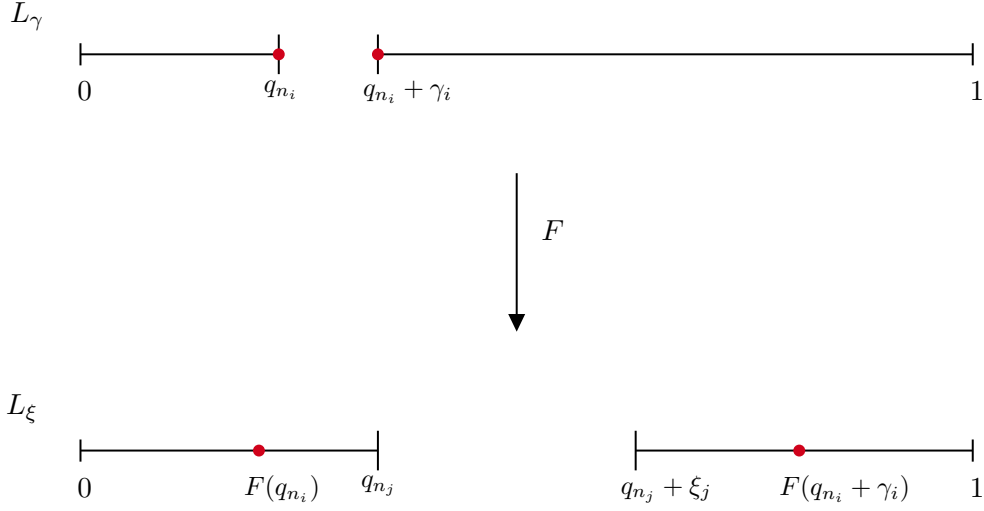
\begin{figure}
\label{examplegapjump}
\begin{tikzpicture}[x=0.75pt,y=0.75pt,yscale=-1,xscale=1]
%uncomment if require: \path (0,476); %set diagram left start at 0, and has height of 476

%Straight Lines [id:da5710842022792811] 
\draw    (250,150) -- (367.5,150) -- (420,150) -- (550,150) ;
\draw [shift={(550,150)}, rotate = 180] [color={rgb, 255:red, 0; green, 0; blue, 0 }  ][line width=0.75]    (0,5.59) -- (0,-5.59)   ;
%Straight Lines [id:da1876478653656629] 
\draw    (100,150) -- (200,150) ;
\draw [shift={(100,150)}, rotate = 180] [color={rgb, 255:red, 0; green, 0; blue, 0 }  ][line width=0.75]    (0,5.59) -- (0,-5.59)   ;
%Straight Lines [id:da7221033848364249] 
\draw    (200,140) -- (200,160) ;
%Straight Lines [id:da7157531590626776] 
\draw    (250,340) -- (250,360) ;
%Straight Lines [id:da08846651427916641] 
\draw    (380,340) -- (380,360) ;
%Straight Lines [id:da9872003443253838] 
\draw    (380,350) -- (550,350) ;
\draw [shift={(550,350)}, rotate = 180] [color={rgb, 255:red, 0; green, 0; blue, 0 }  ][line width=0.75]    (0,5.59) -- (0,-5.59)   ;
%Straight Lines [id:da6244327577359149] 
\draw    (100,350) -- (250,350) ;
\draw [shift={(100,350)}, rotate = 180] [color={rgb, 255:red, 0; green, 0; blue, 0 }  ][line width=0.75]    (0,5.59) -- (0,-5.59)   ;
%Straight Lines [id:da6219037055628542] 
\draw    (190,350) ;
%Straight Lines [id:da5973071211229707] 
\draw    (320,210) -- (320,287) ;
\draw [shift={(320,290)}, rotate = 270] [fill={rgb, 255:red, 0; green, 0; blue, 0 }  ][line width=0.08]  [draw opacity=0] (8.93,-4.29) -- (0,0) -- (8.93,4.29) -- cycle    ;
%Straight Lines [id:da19093705044565912] 
\draw    (250,140) -- (250,160) ;
%Shape: Circle [id:dp712500457336561] 
\draw  [color={rgb, 255:red, 208; green, 2; blue, 27 }  ,draw opacity=1 ][fill={rgb, 255:red, 208; green, 2; blue, 27 }  ,fill opacity=1 ] (197.5,150) .. controls (197.5,148.62) and (198.62,147.5) .. (200,147.5) .. controls (201.38,147.5) and (202.5,148.62) .. (202.5,150) .. controls (202.5,151.38) and (201.38,152.5) .. (200,152.5) .. controls (198.62,152.5) and (197.5,151.38) .. (197.5,150) -- cycle ;
%Shape: Circle [id:dp9901290422121388] 
\draw  [color={rgb, 255:red, 208; green, 2; blue, 27 }  ,draw opacity=1 ][fill={rgb, 255:red, 208; green, 2; blue, 27 }  ,fill opacity=1 ] (247.5,150) .. controls (247.5,148.62) and (248.62,147.5) .. (250,147.5) .. controls (251.38,147.5) and (252.5,148.62) .. (252.5,150) .. controls (252.5,151.38) and (251.38,152.5) .. (250,152.5) .. controls (248.62,152.5) and (247.5,151.38) .. (247.5,150) -- cycle ;
%Shape: Circle [id:dp5647721286505989] 
\draw  [color={rgb, 255:red, 208; green, 2; blue, 27 }  ,draw opacity=1 ][fill={rgb, 255:red, 208; green, 2; blue, 27 }  ,fill opacity=1 ] (187.5,350) .. controls (187.5,348.62) and (188.62,347.5) .. (190,347.5) .. controls (191.38,347.5) and (192.5,348.62) .. (192.5,350) .. controls (192.5,351.38) and (191.38,352.5) .. (190,352.5) .. controls (188.62,352.5) and (187.5,351.38) .. (187.5,350) -- cycle ;
%Shape: Circle [id:dp696023653623686] 
\draw  [color={rgb, 255:red, 208; green, 2; blue, 27 }  ,draw opacity=1 ][fill={rgb, 255:red, 208; green, 2; blue, 27 }  ,fill opacity=1 ] (460,350) .. controls (460,348.62) and (461.12,347.5) .. (462.5,347.5) .. controls (463.88,347.5) and (465,348.62) .. (465,350) .. controls (465,351.38) and (463.88,352.5) .. (462.5,352.5) .. controls (461.12,352.5) and (460,351.38) .. (460,350) -- cycle ;

% Text Node
\draw (63,122) node [anchor=north west][inner sep=0.75pt]   [align=left] {$\displaystyle L_{\gamma }$};
% Text Node
\draw (61,320) node [anchor=north west][inner sep=0.75pt]   [align=left] {$\displaystyle L_{\xi }$};
% Text Node
\draw (191,162) node [anchor=north west][inner sep=0.75pt]  [font=\small] [align=left] {$\displaystyle q_{n_{i}}$};
% Text Node
\draw (241,162) node [anchor=north west][inner sep=0.75pt]  [font=\small] [align=left] {$\displaystyle q_{n_{i}} +\gamma _{i}$};
% Text Node
\draw (240,361) node [anchor=north west][inner sep=0.75pt]  [font=\small] [align=left] {$\displaystyle q_{n_{j}}$};
% Text Node
\draw (371,361) node [anchor=north west][inner sep=0.75pt]  [font=\small] [align=left] {$\displaystyle q_{n_{j}} +\xi _{j}$};
% Text Node
\draw (181,361) node [anchor=north west][inner sep=0.75pt]  [font=\small] [align=left] {$\displaystyle F( q_{n_{i}})$};
% Text Node
\draw (451,361) node [anchor=north west][inner sep=0.75pt]  [font=\small] [align=left] {$\displaystyle F( q_{n_{i}} +\gamma _{i})$};
% Text Node
\draw (331,232.4) node [anchor=north west][inner sep=0.75pt]    {$F$};
% Text Node
\draw (97,162.4) node [anchor=north west][inner sep=0.75pt]    {$0$};
% Text Node
\draw (97,362.4) node [anchor=north west][inner sep=0.75pt]    {$0$};
% Text Node
\draw (547,162.4) node [anchor=north west][inner sep=0.75pt]    {$1$};
% Text Node
\draw (547,362.4) node [anchor=north west][inner sep=0.75pt]    {$1$};

\end{tikzpicture}

\caption{The gap $C^\gamma_i$ jumps over $C^\xi_j$ with respect to $F$.}

\end{figure}

The second lemma we prove can also be easily deduced and is intuitively clear: It shows that if a gap $C^\gamma_i$ in $L_\gamma$ jumps over several gaps in $L_\xi$ simultaneously with respect to a Lipschitz  function $F$, then $\gamma_i$ (the length of $C^\gamma_i$) must be bigger than the length of the smallest subinterval of $[0,1]$ that contains all the gaps $C^\gamma_i$ jumps over, divided by the Lipschitz constant of $F$. 
\begin{lemma}
\label{onlylonggapscanjump}
Let $K>1$, and let $\gamma,\xi\in \Gamma$, with $L_\gamma$ and $L_\xi$ being the corresponding subsets. Suppose that there is a non-decreasing Lipschitz function $F\colon {L_{\gamma}}\rightarrow L_{\xi}$ with $\|F\|_\text{Lip}= K$ such that $F(0)=0$ and $F(1)=1$. Let $C^\gamma_{i_0}$ be a gap in $L_\gamma$, and let $\big(C^\xi_{j}\big)_{j=1}^k$ be a finite collections of different gaps in $L_\xi$. If $C^\gamma_{i_0}$ jumps over $C^\xi_j$ with respect to $F$ for all $1\leq j\leq k$, then 

$$K\gamma_{i_0}\geq \max_{j\neq j'} |q_{n_j}+\xi_j-q_{n_{j'}}|.$$
\end{lemma}
\begin{proof}
Since $C^\gamma_{i_0}$ jumps over $C^\xi_j$ for all $1\leq j\leq k$, we have that $F(q_{n_{i_0}})<q_{n_j}$ and $F(q_{n_{i_0}}+\gamma_{i_0})> q_{n_j}+\xi_j$. Hence, we have that 

$$F(q_{n_{i_0}}+\gamma_{i_0})-F(q_{n_{i_0}})> \max_{j\neq j'} |q_{n_j}+\xi_j-q_{n_{j'}}|.$$
Applying that $F$ is $K$-Lipschitz and $q_{n_{i_0}}+\gamma_{i_0}-q_{n_{i_0}}=\gamma_{i_0}$ we obtain the result.
\end{proof}

We are going to define also a particular kind of intervals which we will be using in the construction of $\gamma^*$. Let $(a,b)\subset [0,1]$ be a nontrivial open interval, and let $r>0$. We define the \textit{sweeping of $[a,b]$ by $r$} as the interval

$$ \mathcal{D}_r(a,b)=(b-r,a+r).$$

Notice that if $r<b-a$, then $\mathcal{D}_r(a,b)=\emptyset$. We can prove some direct properties about this sweeping:

\begin{proposition}
\label{propertiesdilations}
The following properties are verified:

\begin{itemize}
    \item[(1)] Let $(a,b)\subset [0,1]$ be a nontrivial open interval, and let $r>0$. Then $\mu\big(\mathcal{D}_r(a,b)\big)<2r$.
    \item[(2)] Let $r>0$, and $\gamma,\xi\in \Gamma$. Let $F\colon L_\gamma\rightarrow L_\xi$ be a $K$-Lipschitz map. Suppose that there is a gap $C^\gamma_{i_0}$ in $L_\gamma$ that jumps over two gaps $C^\xi_{j_1},C^\xi_{j_2}$ in $L_\xi$ with respect to $F$. Moreover, suppose that $C^\xi_{j_2}\nsubseteq \mathcal{D}_r(C^\xi_{j_1})$. Then $K\gamma_{i_0}>r$. 
\end{itemize}

\end{proposition}
\begin{proof}
Statement $(1)$ is easy to see.
For statement $(2)$, notice that if $C^\xi_{j_2}\nsubseteq \mathcal{D}_r(C^\xi_{j_1})$, this means that either $q_{n_{j_1}}+\xi_{j_1}-r-q_{n_{j_2}}>0$, or $q_{n_{j_2}}+\xi_{j_2}-q_{n_{j_1}}-r>0$. In any case, we obtain that 

$$ \max \{|q_{n_{j_1}}+\xi_{j_1}-q_{n_{j_2}}|,|q_{n_{j_2}}+\xi_{j_2}-q_{n_{j_1}}|\}>r,$$
and the result follows from Lemma \ref{onlylonggapscanjump}.
\end{proof}

Finally we are able to prove the main result about the sets $L_\gamma$:

\begin{proof}[Proof of Theorem \ref{nolipschitzmap}]
We are going to construct inductively by a diagonal method a sequence $\gamma^*=(\gamma_i^*)_{i=1}^\infty\in \Gamma$ with the following properties: 

\begin{itemize}
    \item[(1)] $\gamma^*_i< 2^{-(i+1)}K^{-1}\varepsilon_0$, for all $i\in\mathbb{N}$.
    \item[(2)] $C^{\gamma^*}_i$ is non-empty for every $i\in\mathbb{N}$.
    \item[(3)] If $\gamma\in\Gamma$ is a sequence such that $\gamma_j=\gamma_j^*$ for all $j\leq i$, then there is no Lipschitz map $F\colon L_{\gamma}\rightarrow L_{\gamma^{i}}$ with $\|F\|_\text{Lip}\leq K$ such that $F(0)=0$ and $F(1)=1$. In particular, there is no Lipschitz map $F\colon L_{\gamma^*}\rightarrow L_{\gamma^n}$ with $\|F\|_\text{Lip}\leq K$ such that $F(0)=0$ and $F(1)=1$ for any $n\in\mathbb{N}$.
\end{itemize}

For $i=1$, consider the sequence $\gamma^1$ and define $\gamma_1^*=2^{-2}K^{-1}\varepsilon_0\gamma_1^1$. Let $\gamma=(\gamma_i)_{i=1}^\infty\in \Gamma$ with $\gamma_1=\gamma^*_1$. Suppose by contradiction that there exists a Lipschitz map $F\colon L_{\gamma}\rightarrow L_{\gamma^1}$ with $\|F\|_\text{Lip}\leq K$ and $F(0)=0$ and $F(1)=1$. We assume $F$ to be non-decreasing by Proposition \ref{wlogFisnondecreasing}.

Given the gap $C^{\gamma^1}_1=(q_1,q_1+\gamma^1_1)$, by Lemma \ref{gapjumpsovergap}, there exists a gap $C^\gamma_k=(q_{n_k},q_{n_k}+\gamma_k)$ such that $C^\gamma_k$ jumps over $C^{\gamma^1}_1$. Then, by Lemma \ref{onlylonggapscanjump}, we have that $K\gamma_k>\gamma^1_1$. However, we know that $K\gamma_k\leq K\gamma_1=2^{-2}\varepsilon_0\gamma_1^1<\gamma_1^1$, a contradiction. The first step of the induction is complete.

Suppose we have selected $\{\gamma^*_j\}_{j=1}^i$ verifying the desired properties for $i\in \mathbb{N}$, and consider now the sequence $\gamma^{i+1}$. Let $\sigma=\{j_k\}_{k=1}^i$ be an ordering of the sequence $\{1,\dots,i\}$. Consider $j_1$, put $n_0=1$, and define

$$S_{j_1}=\mathcal{D}_{K\gamma^*_{j_1}}(C^{\gamma^{i+1}}_{n_0}). $$

$S_{j_1}$ is the sweeping of the first gap of $\gamma^{i+1}$ by $K\gamma^*_{j_1}$. The measure of $S_{j_1}$ is at most $2\gamma_{j_1}^0<2^{-j_1}\varepsilon_0$. Hence, since $L_\gamma$ has measure greater than $\varepsilon_0$, the set $L_\gamma\setminus S_{j_1}=L_\gamma\cap ([0,1]\setminus S_{j_1})$ is nonempty. Since $[0,1]\setminus S_{j_1}$ is a finite union of intervals, by Proposition \ref{propertieslgamma} (4), there must exist infinitely many endpoints in $L_\gamma\setminus S_{j_1}$. We can then consider

$$n_{j_1}=\min\{n> n_0\colon ~C^{\gamma^{i+1}}_n\nsubseteq S_{j_1}\}. $$

Intuitively, $C^{\gamma^{i+1}}_{n_{j_1}}$ is the biggest gap of $\gamma^{i+1}$ smaller than $C^{\gamma^{i+1}}_{n_0}$ which is not contained in the sweeping $S_{j_1}$. We continue the process defining

$$ S_{(j_1,j_2)}=S_{j_1}\cup \mathcal{D}_{K\gamma^*_{j_2}}(C^{\gamma^{i+1}}_{n_{j_1}}).$$

The measure of $S_{(j_1,j_2)}$ is at most $(2^{-j_1}+2^{-j_2})\varepsilon_0<\varepsilon_0$, so we can make the same argument as before to find

$$n_{(j_1,j_2)}=\min\{n> n_{j_1}\colon ~C^{\gamma^{i+1}}_n\nsubseteq S_{(j_1,j_2)}\}, $$
which will be the biggest gap of $\gamma^{i+1}$ smaller than $C^{\gamma^{i+1}}_{n_{j_1}}$ not contained in $S_{(j_1,j_2)}$, and thus, not contained in either $\mathcal{D}_{K\gamma^*_{j_1}}(C^{\gamma^{i+1}}_{n_0})$ nor $\mathcal{D}_{K\gamma^*_{j_2}}(C^{\gamma^{i+1}}_{n_{(j_1,j_2)}})$. 

Repeating this process $i$ times, we can define $n_{\sigma}=n_{(j_1,\dots,j_i)}\in \mathbb{N}$ such that $C_{n_{\sigma}}^{\gamma_{i+1}}$ is the biggest gap of $\gamma^{i+1}$ not contained in $\mathcal{D}_{K\gamma^*_{j_k}}(C^{\gamma^{i+1}}_{n_{(j_1,\dots,j_k)}})$ for any $1\leq k\leq i$, and smaller than $C^{\gamma^{i+1}}_{n_{(j_1,\dots,j_k)}}$ for every $1\leq k\leq i$. Notice that this last condition can be written as:

\begin{equation}
\label{gammasigmaisthesmallest}
\gamma^{i+1}_{n_\sigma}< \min_{1\leq k\leq i} \gamma^{i+1}_{n_{(j_1,\dots,j_k)}}.
\end{equation}

Now, let $\Omega=\{\sigma=\{j_k\}_{k=1}^i\colon \sigma\text{ is and ordering of }\{1,\dots,i\}\}$. Clearly $\Omega$ is a finite set, so we can define $n_\Omega=\max\{n_\sigma\colon \sigma\in \Omega\}$. The corresponding gap $C^{\gamma_{i+1}}_{n_\Omega}$ is smaller than or equal to each $C_{n_{\sigma}}^{\gamma_{i+1}}$. Equivalently, we have that 
\begin{equation}
\label{gammaomegaisthesmallest}
\gamma^{i+1}_{n_\Omega}\leq \min_{\sigma\in \Omega}\gamma^{i+1}_{n_\sigma}.
\end{equation}

Finally, define $\gamma^{*}_{i+1}<2^{-(i+2)}K^{-1}\varepsilon_0\gamma_{n_\Omega}^{i+1}$. We can take it small enough so that $C^{\gamma^*}_{i+1}$ is non-empty. Let $\gamma=(\gamma_i)_{i=1}^\infty\in \Gamma$ with $\gamma_j=\gamma^*_j$ for $1\leq j\leq i+1$. Suppose by contradiction that there exists a Lipschitz map $F\colon L_{\gamma}\rightarrow L_{\gamma^{i+1}}$ with $\|F\|_\text{Lip}\leq K$ and $F(0)=0$ and $F(1)=1$. Again, we may assume that $F$ is non-decreasing.

Put $n_0=1$, and consider the gap $C_{n_0}^{\gamma^{i+1}}$. By Lemma \ref{gapjumpsovergap}, there exists a $j_1\in\mathbb{N}$ such that $C^\gamma_{j_1}$ jumps over $C_{n_0}^{\gamma^{i+1}}$. Since $K\gamma_{i+1}<\gamma^{i+1}_{n_\Omega}<\gamma^{i+1}_1$, by Lemma \ref{onlylonggapscanjump} we have that $j_1\leq i$. 

Consider now the gap $C^{\gamma^{i+1}}_{n_{j_1}}$, and take $j_2$ such that $C^{\gamma}_{j_2}$ jumps over $C^{\gamma^{i+1}}_{n_{j_1}}$. Again, since $K\gamma_{i+1}<\gamma_{n_\Omega}<\gamma^{i+1}_{n_{j_1}}$, we obtain that $j_2\leq i$. Moreover, $j_2$ is different from $j_1$. Indeed, if $j_2=j_1$, then $C^{\gamma}_{j_1}$ jumps over $C_{n_0}^{\gamma^{i+1}}$ and $C^{\gamma^{i+1}}_{n_{j_1}}$. By the choice of $n_{j_1}$, $C^{\gamma^{i+1}}_{n_{j_1}}\nsubseteq \mathcal{D}_{K\gamma^*_{j_1}}(C^{\gamma^{i+1}}_{n_0})$, so by Proposition \ref{propertiesdilations}, we have that $K\gamma_{j_1}>K\gamma^*_{j_1}$, a contradiction.

We can repeat this process $i$ times until we obtain a sequence $\sigma=\{j_k\}_{k=1}^i$ of $i$ different numbers in $\{1,\dots,i\}$ (so $\sigma\in \Omega$) such that $C_{j_{k}}^\gamma$ jumps over $C^{\gamma^{i+1}}_{n_{(j_1,\dots,j_{k})}}$ for all $1\leq k\leq i$. To finish the proof, consider the gap $C^{\gamma^{i+1}}_{n_\sigma}$, where $n_\sigma$ is as defined above for $\sigma\in \Omega$, and take $\tilde{i}$ such that $C^{\gamma}_{\tilde{i}}$ jumps over $C^{\gamma^{i+1}}_{n_\sigma}$. Reasoning as before, by choice of $\gamma^*_{i+1}$ and using equation \eqref{gammaomegaisthesmallest} we have that $\tilde{i}\leq i$. Then $\tilde{i}=j_{k_0}$ for some $1\leq k_0\leq i$. We have chosen $k_0$ such that $C_{j_{k_0}}^\gamma$ jumps over $C^{\gamma^{i+1}}_{n_{(j_1,\dots,j_{k_0})}}$. Moreover, $C^{\gamma^{i+1}}_{n_\sigma}$ is not contained in $\mathcal{D}_{K\gamma^{0}_{j_{k_0}}}\big( C^{\gamma^{i+1}}_{n_{(j_1,\dots,j_{k_0})}}\big)$. Therefore, by Proposition \ref{propertiesdilations}, we have that $K\gamma_{j_{k_0}}>K\gamma^*_{j_{k_0}}$, a contradiction. 

\end{proof}
\subsection{Construction of the metric space without Lipschitz SRP}
The rest of this section is dedicated to constructing the complete metric space that does not admit a Lipschitz retraction onto any separable subset containing two particular fixed points, for which we will be using the previous result. 

Let $\widehat{L}_{\gamma}=\{(p,\gamma)\colon p\in L_\gamma,~p\neq 0,~p\neq 1\}$ for each $\gamma\in \Gamma$. Define $\widehat{M}=\bigcup_{\gamma\in \Gamma} \widehat{L}_\gamma$. This set is the disjoint union of all $L_\gamma$ minus the points $\{0,1\}$ for each $\gamma\in\Gamma$. We consider this disjoint union because we want to construct $M$ in such a way that the different $L_\gamma$ only have in common the points $0$ and $1$. Hence, we define $M=\{0,1\}\cup \widehat{M}$. We endow $M$ with the metric $d$, defined by:

$$ 
\small
d(p,q)=
\begin{cases}
1,&\text{ if } p=1,q=0,\\
p', &\text{ if }q=0,~ p=(p',\gamma),\text{ for } p'\in L_\gamma,~ \gamma\in \Gamma,\\
1-p', &\text{ if }q=1,~p=(p',\gamma),\text{ for }p'\in L_\gamma,~ \gamma\in \Gamma,\\
|p-q|, &\text{ if }p=(p',\gamma),q=(q',\gamma),\text{ for }p',q'\in L_\gamma,~ \gamma\in \Gamma,\\
H(p,q),&\text{ if }p=(p',\gamma_1),q=(q',\gamma_2),\text{ for }p'\in L_{\gamma_1},q'\in L_{\gamma_2},~ \gamma_1\neq\gamma_2\in \Gamma,
\end{cases}
$$
where $H(p,q)=\min\big\{d(p,0)+d(q,0),d(p,1)+d(q,1)\big\}$ for any $p,q\in M$.
It is straightforward to see that this does in fact define a complete metric. Moreover, considering the isometry $I_\gamma\colon L_\gamma\rightarrow \widehat{L}_\gamma\cup\{0,1\}$ given by $I(0)=0$, $I(1)=1$ and $I(p)=(p,\gamma)$, we have that $L_\gamma$ is isometrically embedded in $M$ for all $\gamma\in\Gamma$. 

We can finally prove the main result of the section:

\begin{theorem}
Let $(M,d)$ be the metric space as defined above. Then no separable subset of $M$ containing $0$ and $1$ is a Lipschitz retract of $M$. In particular, $M$ fails the Lipschitz SRP.
\end{theorem}
\begin{proof}
Let $S\subset M$ be a separable subset with $0,1\in S$. Since it is separable, $S$ intersects $\widehat{L}_\gamma$ in a non-empty set for only countably many $\gamma\in\Gamma$. Let $\{\gamma^n\}_{n=1}^\infty$ be a sequence in $\Gamma$ such that $S\subset \bigcup_{n\in\mathbb{N}}L_{\gamma^n}$. Suppose that there exists $K>1$ and a Lipschitz retraction $R\colon M\rightarrow S$ with $\|R\|_\text{Lip}\leq K$. Using Theorem \ref{nolipschitzmap}, let $\gamma^*\in \Gamma$ be a sequence such that there is no $K$-Lipschitz function $F\colon {L_{\gamma^*}}\rightarrow L_{\gamma^n}$ with $F(0)=0$ and $F(1)=1$ for any $n\in\mathbb{N}$. 

We can restrict $R$ to $\widehat{L}_{\gamma^*}\cup\{0,1\}\subset M$, which is isometric to $L_{\gamma^*}$ by the map $I_{\gamma^*}$ we defined above. The resulting map $R_0=R\circ I_{\gamma^*}\colon L_{\gamma^*}\rightarrow S$ has Lipschitz constant $K$ as well. If we manage to restrict the image of $R_0$ to a single $L_{\gamma^{n_0}}$ without increasing the Lipschitz constant we will reach a contradiction. Therefore, the rest of the proof will be dedicated to defining a $K$-Lipschitz map from $L_{\gamma^*}$ to a particular $L_{\gamma^{n_0}}$ for some $n_0\in\mathbb{N}$ that fixes $0$ and $1$, using the map $R_0$ as a starting point.

Consider the following set:
\begin{align*}
    \Delta_0= \big\{x\in L_{\gamma^*}\colon &\exists n_x\in\mathbb{N},~ \exists y_x\in L_{\gamma^*},~y_x\geq x, \text{ such that } \\
    & R_0\big([x,y_x)\big)\subset \widehat{L}_{\gamma^{n_x}}\cup\{0,1\},\text{ and }\\
    & d\big(R_{0}(z),1\big)+d\big(R_0(y_x),1\big)\leq K(y_x-z),~\forall z\in [x,y_x)\big\}
\end{align*}

The point $1\in L_{\gamma^*}$ trivially verifies $1\in \Delta_0$, so the infimum $P=\inf \Delta_0$ exists. Moreover, if $x\in L_{\gamma^*}$ verifies the condition in $\Delta_0$ with $n_x\in\mathbb{N}$ and $y_x\in L_{\gamma^*}$, then every $z\in [x,y_x)$ also verifies it with $n_z=n_x$ and $y_z=y_x$. We can use this fact and the compactness of $L_{\gamma^*}$ to deduce that $P\in \Delta_0$ (the infimum is actually a minimum). 

Let us first prove that $R_0(P)\neq 1$. Indeed, suppose that $R_0(P)=1$. First, we show that in that case we have that 

\begin{equation}
\label{gapbehindP}
\bigg(P-\frac{2}{3K},P\bigg)\cap L_{\gamma^*}=\emptyset. 
\end{equation}

By contradiction, suppose there exists an $x\in L_{\gamma^*}$ with $0<P-x<2/(3K)$. If $R_0(x)=0$, then 
$$d\big(R_0(x),R_0(P)\big)=1>2/3>K(P-x),$$
a contradiction. If $R_0(z)=1$ for all $z\in [x,P)\cap L_{\gamma^*}$, then $x\in \Delta_0$, which contradicts the minimality of $P$. Hence, for some $z\in [x,P)\cap L_{\gamma^*}$, we have $R_0(z)\neq 0$ and $R_0(z)\neq 1$. Without loss of generality, we can assume that $x$ verifies this property itself, so $R_0(x)\in \widehat{L}_{\gamma^{n_1}}$ for some $n_1\in\mathbb{N}$. Since $x<P$ and by the minimality of $P$, there exists a $y\in L_{\gamma^*}$ with $y\in (x,P)$ and $R_0(y)\notin \widehat{L}_{\gamma^{n_1}}\cup\{0,1\}$ (otherwise $x\in \Delta_0$ with $n_x=n_1$ and $y_x=P$). We can define then:

$$y_0=\min\{y\in L_{\gamma^*}\colon y\in (x,P),~\text{and } R_0(y)\notin \widehat{L}_{\gamma^{n_1}} \}.$$
Clearly $R_0\big([x,y_0)\big)\subset \widehat{L}_{\gamma^{n_1}}\cup\{0,1\}$ and $R_0(y_0)\notin \widehat{L}_{\gamma^{n_1}}$, so for any $z\in [x,y_0)$ we have:

$$d(R_0(z),R_0(y_0))=\min\{d(R_0(z),0)+d(R_0(y_0),0),d(R_0(z),1)+d(R_0(y_0),1)\}.$$

The fact that $d\big(R_0(z),R_0(y_0)\big)=d\big(R_0(z),1\big)+d\big(R_0(y_0),1\big)$ for all $z\in [x,y_0)$ contradicts the minimality of $P$. Hence, there exists a $z_0\in [x,y_0)$ such that $d\big(R_0(z),R_0(y_0)\big)=d\big(R_0(z),0\big)+d\big(R_0(y_0),0\big)$. However, on the one hand we have that $d\big(R_0(z),R_0(y_0)\big)\leq K(z-y_0)<2/3$, and on the other hand:
\begin{align*}
    d\big(R_0(z),0\big)+d\big(R_0(y_0),0\big)=2-d\big(R_0(z),1\big)-d\big(R_0(y_0),1\big)>2-4/3=2/3,
\end{align*}
a contradiction. Hence, we have proven equation \eqref{gapbehindP}. Define now 
$$P_0=\max\{x\in L_{\gamma^*}\colon x<P\},$$
which exists by compactness of $L_{\gamma^*}$ and the fact that $\big(P-2/(3K),P\big)\cap L_{\gamma^*}=\emptyset$. The point $P_0$ verifies $P_0<P$ and $(P_0,P)\cap L_{\gamma^*}=\emptyset$. Then $P_0\in L_{\gamma^*}\in \Delta_0$ trivially with $y_{P_0}=P$. However, this again leads to a contradiction with the choice of $P$. We conclude then that $R_0(P)\neq 1$.

Let us prove now that we can choose $P,Q\in L_{\gamma^*}$ with $P\leq Q$ and $n_0\in\mathbb{N}$ such that 
\begin{itemize}
    \item[(1)] $R_0\big([P,Q)\big)\subset \widehat{L}_{\gamma^{n_0}}\cup\{0,1\}$,
    \item[(2)] $d(R_{0}(x),1)+d(R_0(Q),1)\leq K(Q-x),~\forall x\in [P,Q)$, and
    \item[(3)] If $x\in L_{\gamma^*}$ with $x<P$, then there exists $z\in L_{\gamma^*}$ with $x\leq z<P$ such that $R_0(z)\notin \widehat{L}_{\gamma^{n_0}}\cup\{0,1\}$.
    \item[(4)] If $R_0(P)\neq 0$, there exists $P_0\in L_{\gamma^*}$ with $P_0<P$ such that $L_{\gamma^*}\cap(P_0,P)=\emptyset$ and  $d(R_0(P),0)+d(0,R_0(P_0))=d(R_0(P),R_0(P_0))$.
\end{itemize}

Indeed, the first two properties follow from the definition of $\Delta_0$ and the choice of $P$. For the third, notice that if $x\in L_{\gamma^*}$ with $x<P$, by the minimality of $P$ we have necessarily that there exists $z\in L_{\gamma^*}$ with $x\leq z<P$ such that either $R_0(z)\notin \widehat{L}_{\gamma^{n_0}}\cup\{0,1\}$ or $d(R_{0}(z),1)+d(R_0(Q),1)>K(Q-z)$. Suppose that $R_0(z)\in \widehat{L}_{\gamma^{n_0}}\cup\{0,1\}$. Then $d\big(R_{0}(z),1\big)+d\big(R_0(Q),1\big)>K(Q-z)$. 

If $d(R_0(z),1)\leq d(R_0(P),1)$, then 
\begin{align*}
    d(R_0(z),1)+d(1,R_0(Q))&\leq d(R_0(P),1)+d(1,R_0(Q))\\
    &\leq K(Q-P)< K(Q-z),
\end{align*}
a contradiction. Then $d(R_0(z),1)> d(R_0(P),1)$, and we have that $d(R_0(z),1)=d(R_0(z),R_0(P))+d(R_0(P),1)$ because they belong to the same $\widehat{L}_{\gamma^{n_0}}$. This implies that

\begin{align*}
    d(R_0(z),1)+d(R_0(Q),1)&=d(R_0(z),R_0(P))+d(R_0(P),1)+d(R_0(Q),1)\\
    &\leq K(P-z)+K(Q-P)=K(Q-z),
\end{align*}
which again contradicts the choice of $z$, and property (3) follows.

For the fourth property, since $R_0(P)$ is neither $0$ nor $1$, we have that $R_0(P)\in \widehat{L}_{\gamma^{n_0}}$, and we can choose $\delta>0$ such that $K\delta<\min\{d\big(R_0(P),0\big),d\big(R_0(P),1\big)\}$. If $x\in L_{\gamma^*}$ verifies $P-x<\delta$, then $d\big(R_0(P),R_0(x)\big)<K\delta$, and thus $x\in L_{\gamma^{n_0}}$. However, this contradicts property (3). Hence, $(P-\delta,P)\cap L_{\gamma^*}=\emptyset$. Consider
$$P_0=\max\{x\in L_{\gamma^*}\colon x<P-\delta\}, $$
which exists by compactness of $L_{\gamma^*}$. Clearly $(P_0,P)\cap L_{\gamma^*}=\emptyset$, so by the property (3) we have that $R_0(P_0)\notin \widehat{L}_{\gamma^{n_x}}\cup\{0,1\}$. This implies that $R_0(P)$ and $R_0(P_0)$ belong to different $\widehat{L}_\gamma$, so 

$$d\big(R_0(P),R_0(P_0)\big)=\min\{d\big(R_0(P),0\big)+d\big(R_0(P_0),0\big),d\big(R_0(P),1\big)+d\big(R_0(P_0),1\big)\}.$$ 

Then $d\big(R_0(P),R_0(P_0)\big)=d\big(R_0(P),0\big)+d\big(R_0(P_0),0\big)$, because $d\big(R_0(P),R_0(P_0)\big)=d\big(R_0(P),1\big)+d\big(R_0(P_0),1\big)$ contradicts the minimality of $P$. We have proven properties $(1)-(4)$.

Finally, define the map $F_0\colon L_{\gamma^*}\rightarrow L_{\gamma^{n_0}}$ by 
$$ 
F_0(x)=
\begin{cases}
0,&\text{ if } x\in L_{\gamma^*},~ x< P,\\
I^{-1}_{\gamma^{n_0}}\big(R_0(x)\big), &\text{ if }x\in [P,Q)\cap L_{\gamma^*},\\
1, &\text{ if } x\in L_{\gamma^*},~ x\geq Q.
\end{cases}
$$
This map is well defined by the choice of $n_0$, and verifies that $F_0(0)=0$ and $F_0(1)=1$. Let us prove that it is $K$-Lipschitz, which will be a contradiction with the choice of $\gamma^*$. Notice that we only need to prove that $|F_0(x)-F_0(y)|\leq K(x-y)$ in two cases:

\begin{itemize}
    \item[(i)] If $x\in L_{\gamma^*}$ with $x<P$ and $y=P$.
    \item[(ii)] If $x\in [P,Q)\cap L_{\gamma^*}$ and $y=Q$.
\end{itemize}

The rest of possibilities are either trivial or can be easily deduced from the above two. Suppose then first that $x\in L_{\gamma^*}$ with $x<P$ and $y=P$. If $R_0(P)=0$ then the desired inequality follows trivially. Otherwise, by property $(4)$ we can find $x\leq P_0<P$ such that $d(R_0(P),0)+d(0,R_0(P_0))=d(R_0(P),R_0(P_0))$. Therefore:

\begin{align*}
    |F_0(y)-F_0(x)|&=d(R_0(P),0)<d(R_0(P),0)+d(0,R_0(P_0))=d(R_0(P),R_0(P_0))\\
    &\leq K(P-P_0)\leq K(y-x).
\end{align*} 

Finally, suppose $x\in [P,Q)\cap L_{\gamma^*}$ and $y=Q$. Then, by the property $(2)$ we have that 

$$|F_0(y)-F_0(x)|=d(R_0(x),1)<d(R_0(x),1)+d(1,R_0(y))\leq K(y-x). $$
Therefore $F_0\colon L_{\gamma^*}\rightarrow L_{\gamma^{n_0}}$ is a Lipschitz map that fixes $0$ and $1$ with $\|F_0\|_\text{Lip}\leq K$, which contradicts the choice of $\gamma^*$, and we are done. 
\end{proof}

\begin{remark}
It is worth emphasizing that in order to construct this metric space, we rely heavily on the non-connectedness of the subsets $L_\gamma$, so we are not able to translate all the techniques we displayed here to the Banach space case. As mentioned in the introduction, finding a Banach space failing the Lipschitz SRP would solve in the negative the long standing conjecture of Benyamini and Lindenstrauss of whether every separable Banach space is a Lipschitz retract of its bidual. The nonseparable case was already proven to fail by Kalton in \cite{Kal11}.

Likewise, the metric space $M$ we defined does have the weaker $\mathcal{F}$-Lipschitz SRP. In fact, $\mathcal{F}(M)$ admits a commutative $1$-projectional skeleton, and it is thus $1$-Plichko. We sketch the proof of this fact: 

Notice that for every closed subset $A$ of $[0,1]$, the map $L\colon [0,1]\rightarrow \mathcal{F}(\{0,1\})$ given by $L(x)=x\delta(1)$ is an $\mathcal{F}$-Lipschitz retraction with $\|L\|_\text{Lip}=1$. 

Set $I=[\Gamma]^{\leq\omega}$, that is, the collection of all countable subsets of $\Gamma$, partially ordered by inclusion. Then $I$ is directed and $\sigma$-complete. For each $C\in I$, we define 

$$A_C=\bigg(\bigcup_{\gamma\in C}\widehat{L}_\gamma\bigg)\cup\{0,1\}. $$
Then $A_C$ is separable for each $C\in I$. Finally, we can define the map $L_C\colon M\rightarrow\mathcal{F}(A_C)$ by $L_C\big((x,\gamma)\big)=\delta(x,\gamma)$ if $\gamma\in C$, and $L_C\big((x,\gamma)\big)=x\delta(1)$ otherwise. Again, this defines an $\mathcal{F}$-Lipschitz retraction with $\|L_C\|_\text{Lip}=1$ for each $C\in I$. By the universal property of Lipschitz free spaces, we can extend these maps to norm $1$ projections $P_C\colon \mathcal{F}(M)\rightarrow \mathcal{F}(A_C)$. It is routine now to check that $\{P_C\}_{C\in I}$ defines a commutative $1$-projectional skeleton in $\mathcal{F}(M)$.

\end{remark}

\section{Metric space failing Lipschitz RP($\Lambda,\Lambda$)}
In the previous section we have constructed a metric space which fails the Lipschitz SRP in a strong sense. In this section we prove that for every infinite cardinal number $\Lambda$, we can find a metric space that fails the Lipschitz RP($\Lambda,\Lambda$). Given a cardinal $\Lambda$, we want to construct a metric space $M$ such that there exists a subset with density character $\Lambda$ which is not contained in any subset of $M$ with density character $\Lambda$ that is a Lipschitz retraction of $M$. 

Set $\Gamma=\{\gamma=(\gamma_\alpha)_{\alpha\in\Lambda}\colon 0<\gamma_\alpha<1/2,~ \forall \alpha\in \Lambda\}$. For every $\gamma\in \Gamma$ we are going to define a subset $M_\gamma\subset [0,1]^\Lambda\subset\ell_\infty(\Lambda)$ in the following way:

$$M_\gamma=\{(p_\alpha)_{\alpha\in \Lambda}\in[0,1]^\Lambda\colon p_\alpha\in[0,1/2-\gamma_\alpha]\cup[1/2+\gamma_\alpha,1],~ \forall\alpha\in\Lambda\}, $$
endowed with the metric inherited from $\ell_\infty(\Lambda)$. Notice that if for a subset $A\subset \Lambda$ we write the point $e_A=((e_A)_\alpha)_{\alpha\in\Lambda}$ (called a \textit{vertex}) as the point such that $(e_A)_\alpha=1$ if $\alpha\in A$, and $(e_A)_\alpha=0$ if $\alpha\notin A$; then $e_A\in M_\gamma$, for any choice of $A\subset \Lambda$ and $\gamma\in \Gamma$. If $A=\{\alpha\}$ is a singleton, we write $e_{\{\alpha\}}=e_\alpha$. Notice also that $e_\emptyset = 0\in M$. 

As we did with the previous example, set for each $\gamma\in \Gamma$:

$$\widehat{M}_\gamma=\{(p,\gamma)\colon p\in M_\gamma,~ p\neq e_A,~\text{for any}A\subset \Lambda\},$$
and consider 
$$M=\bigg(\bigcup_{\gamma\in\Gamma}\widehat{M}_\gamma\bigg)\cup \{e_A\}_{A\subset \Lambda}.$$
Alternatively, the set $M$ can be realized by considering the disjoint union of each $M_\gamma$ and then identifying each vertex $e_A$ with its corresponding copy in every $M_\gamma$. 
We will define a metric $d$ on $M$ ``step-by-step". Let $p,q\in M$. If $p,q\in \widehat{M}_\gamma\cup \{e_A\}_{A\subset \Lambda}$ for a fixed $\gamma\in\Gamma$, then 

$$d(p,q)=\|p-q\|_\infty, $$
where we make the identification $p=(p,\gamma)\in \widehat{M}_\gamma$ for any point in $\widehat{M}_\gamma$. If $p,q\in M$ belong to different $\widehat{M}_{\gamma_1},\widehat{M}_{\gamma_2}$ respectively, then 

$$d(p,q)=\inf_{A\subset \Lambda} \{d(p,e_A)+d(e_A,q)\}. $$
Notice that if a point $p$ is not a vertex, then there exists a coordinate $\alpha\in \Lambda$ such that $0<p_\alpha<1$, and so $d(p,e_A)\geq\min\{p_\alpha,1-p_\alpha\}>0$ for every $A\subset \Lambda$. This shows that $d(p,q)=0$ if and only if $p=q$. The triangle inequality follows directly from the definition of the metric $d$.

Also note that for any $p\in M$ and any $\alpha\in \Gamma$, the coordinate $p_\alpha$ cannot be equal to $1/2$ by construction of $M$.

\subsection{Arc-connected components of $M$}
The metric space $M$ as defined above is not arc-connected. Indeed, if we define an arc between two points $p,q\in M$ as a continuous map $F\colon [a,b]\rightarrow M$ with $a<b$, such that $F(a)=p$ and $F(a)=q$, then, for instance, the points $e_A$ and $e_B$ are not connected by an arc if $A\neq B\subset \Lambda$. We will prove this in detail in this section. Let us first define precisely the concepts we will be using:

We say that two points $p,q$ are \textit{arc-connected} if there exists an arc between $p$ and $q$. This defines an equivalence relation in $M$. Moreover, if $F$ is an arc that connects $p$ and $q$, then it is straightforward to see that $p$ is arc-connected with any point in $F([a,b])$. Therefore, given $C\subset M$ an equivalence class of this relation in $M$, we have that any two points in $C$ are connected by an arc whose range is contained in $C$. We call the equivalence classes the \textit{arc-connected components of $M$}, and they form a partition of $M$. $M$ is said to be \textit{arc-connected} if $M$ is the only equivalence class. 

Note that if $p,q\in M$ are connected by an arc $F\colon [a,b]\rightarrow M$, we can assume without loss of generality that $a=0$ and $b=1$. 
\begin{lemma}
\label{ArcconnectingdifferentMgammasgoesthroughvertex}
Let $p,q\in M$ be two arc-connected points in $M$ such that $p\in M_{\gamma_1}$ and $q\in M_{\gamma_2}$ with $\gamma_1\neq \gamma_2\in \Gamma$. Then, for every arc $F\colon [0,1]\rightarrow M$ connecting $p$ and $q$ there exists $t_0\in (0,1)$ and $A\subset \Lambda$ such that $F(t_0)=e_A$. 
\end{lemma}
\begin{proof}
Consider 

$$t_0=\min\{t\in [0,1]\colon F(t)\notin \widehat{M}_{\gamma_1}\}, $$
which exists by continuity of $F$ and the fact that $F(1)=q\notin \widehat{M}_{\gamma_1}$. We claim that $F(t_0)$ is a vertex. Indeed, suppose there exists $\gamma_0\in \Gamma$ such that $F(t_0)\in \widehat{M}_{\gamma_0}$. By definition of $t_0$, we have that $\gamma_0\neq \gamma_1$. The set $\widehat{M}_{\gamma_0}$ is open in $M$, so by continuity of $F$, there exists $\varepsilon>0$ such that $F((t_0-\varepsilon,t_0+\varepsilon))\subset \widehat{M}_{\gamma_0}$. However, this contradicts the minimality of $t_0$. Therefore $F(t_0)=e_A$ for some $A\subset \Lambda$.
\end{proof}
\begin{proposition}
\label{Connectedcomponents}
Let $M$ be the metric space as defined above for the cardinal $\Lambda$. Then, for each $A\neq B\subset \Lambda$, the points $e_A$ and $e_B$ are inside different arc-connected components of $M$.
\end{proposition}
\begin{proof}
Suppose there is an arc $F\colon [0,1]\rightarrow M$ with $F(0)=e_A$ and $F(1)=e_B$. Consider the following points:

\begin{align*}
    a_0 &= \max\{t\in [0,1]\colon F(t)=e_A\}\\
    b_0 &= \min\{t\in [0,1]\colon F(t)=e_B\}\\
\end{align*}
which exist by continuity of $F$. There are two possibilities: either there exists $\gamma_0\in \Gamma$ such that $F(a_0,b_0)\subset \widehat{M}_{\gamma_0}$, or there exist $t_1<t_2\in (a_0,b_0)$ and $\gamma_1\neq \gamma_2\in \Gamma$ such that $F(t_1)\in \widehat{M}_{\gamma_1}$ and $F(t_2)\in \widehat{M}_{\gamma_2}$. 

Notice that in the second case, the restriction of $F$ to $[t_1,t_2]$ forms an arc between $F(t_1)$ and $F(t_2)$, so by the previous lemma there is $r\in (t_1,t_2)$ and $C\subset \Lambda$ such that $F(r)=e_{C}$. Taking the minimum over all such $r\in (a_0,b_0)$ yields a point $b_0^*\in (a_0,b_0)$ and $B^*\subset \Lambda$ such that $F(b_0^*)=e_{B^*}$ and $F(a_0,b_0^*)$ is contained in a $\widehat{M}_{\gamma^*_0}$ for some $\gamma^*_0\in \Gamma$. Moreover, $B^*\neq A$ by maximality of $a_0$.

In either case, without loss of generality we can assume that $F(0,1)\subset \widehat{M}_{\gamma_0}$ for some $\gamma_0\in \Gamma$. Since $A\neq B$, we may assume without loss of generality that there exists an $\alpha\in B\setminus A$. Then $(e_A)_\alpha=0$ and $(e_B)_\alpha=1$. For each $\alpha\in \Lambda$, the projection $P^{\gamma_0}_\alpha\colon \widehat{M}_{\gamma_0}\cup\{e_A\colon A\subset \Lambda\}\rightarrow [0,1]$ given by $P_\alpha(p)=p_\alpha$, is a continuous map. Therefore, the composition map $\overline{F}=P_\alpha\circ F\colon [0,1]\rightarrow[0,1]$ is continuous too, and verifies that $\overline{F}(0)=0$ and $\overline{F}(1)=1$. Therefore, there exists $t^*\in (0,1)$ such that $\overline{F}(t^*)=1/2$. However, this means that the $\alpha$-th coordinate of $F(t^*)\in \widehat{M}_{\gamma^*}$ is equal to $1/2$, which is a contraction. 
\end{proof}
Thanks to this last result, we can properly define for each $A\subset \Lambda$ the arc-connected component $C_A$ to be the arc-connected component of $M$ that contains the vertex $e_A$. Moreover, given a point $p\in \widehat{M}_\gamma$ for some $\gamma\in \Gamma$, it is straightforward to see that there exists a (unique) vertex $e_A$ such that $p\in C_A$. Indeed, we have that 

$$C_A=\{(p,\gamma)\in M\colon |p_\alpha-(e_A)_\alpha|<1/2,~\forall \alpha\in\Lambda,~\gamma\in\Gamma\}\cup\{e_A\}. $$

For convenience, we write $C_{\{\alpha\}}=C_\alpha$ for every $\alpha\in\Lambda$, and $C_\emptyset= C_0$.
\subsection{Non-existence of Lipschitz retracts of cardinality $\Lambda$ containing a set of vertices}
We now prove that certain subsets of $M$ with density character $\Lambda$ are not contained in any subset of the same density character which is a Lipschitz retract of $M$. 

\begin{theorem}
\label{Proofofcounterexampleforarbitratydensities}
Let $\Lambda$ be an infinite cardinal. There exists a metric space $M$ and a subspace $N\subset M$ with density character $\Lambda$ such that every intermediate subset containing $N$ with density character $\Lambda$ is not a Lipschitz retract of $M$.
\end{theorem}
\begin{proof}
Let $M$ be the metric space we have defined in this section associated with the cardinal $\Lambda$. Put 
$$N= \bigg(\bigcup_{\alpha\in\Lambda}e_{\alpha}\bigg)\cup \{0\}, $$
which clearly verifies $\text{dens}(N)=\Lambda$. Let $S$ be a subset of $M$ such that $A\subset S$ and $\text{dens}(S)=\Lambda$, and let $K\geq 1$. We are going to prove that $S$ is not a $K$-Lipschitz retract of $M$. 

Since $\text{dens}(S)=\Lambda$, there exists a subset $\Gamma'\subset \Gamma$ with $\text{card}(\Gamma')=\Lambda$ such that 
$$S\subset \bigg(\bigcup_{\gamma\in \Gamma'}\widehat{M}_{\gamma}\bigg)\cup \{e_A\colon A\subset \Lambda\}. $$

We write $\Gamma'=(\gamma^\beta)_{\beta\in \Lambda}$ and $\gamma^\beta=(\gamma^\beta_\alpha)_{\alpha\in\Lambda}$ for $\beta\in \Lambda$.

Define $\gamma^*\in \Gamma$ as follows: $\gamma^*_\alpha = (2K)^{-1}\gamma^\alpha_\alpha$ for each $\alpha\in\Lambda$. We obtain directly that $\gamma^*\notin\Gamma'$. Now, suppose that $F\colon M\rightarrow S$ is a $K$-Lipschitz retraction onto $S$. The image of an arc-connected set under a continuous function is still arc-connected, so in particular, $F(C_\alpha)\subset C_\alpha$, since we know that $F(e_\alpha)=e_\alpha$ for every $\alpha\in\Lambda$. By the same argument, we have that $F(C_0)\subset C_0$. 

Consider now the point $p^*=\bigg(\big(1/2-\gamma^*_\alpha\big)_{\alpha\in \Lambda},\gamma^*\bigg)\in \widehat{M}_{\gamma^*}$. Since each coordinate of $p^*$ is less than $1/2$, we have that $p^*\in C_0$, which in turn implies that $F(p^*)\in C_0$. There are two possibilities: either $F(p^*)=0$, or there exists a $\beta_0\in\Lambda$ such that $F(p^*)\in C_0\cap \widehat{M}_{\gamma^{\beta_0}}$.

Suppose first that $F(p^*)=0$. Choose any $\beta_0\in \Lambda$. The point $q^*=\big(q,\gamma^*\big)\in \widehat{M}_{\gamma^*}$ defined by $q_\alpha= 1/2-\gamma^*_\alpha$ if $\alpha\neq \beta_0$ and $q_{\beta_0}=1/2+\gamma^*_{\beta_0}$ verifies $q^*\in C_{\beta_0}$, so $F(q^*)\subset C_{\beta_0}\cap S$. However, now we have that $d(p^*,q^*)=2\gamma^*_{\beta_0}<K^{-1}\gamma^{\beta_0}_{\beta_0}<K^{-1}1/2$, and on the other hand

$$d(F(p^*),F(q^*))> d(0,F(q^*))>1/2,$$
which contradicts the fact that $F$ is $K$-Lipschitz.

Suppose now that there exists a $\beta_0\in \Lambda$ such that $F(p^*)\in C_0\cap \widehat{M}_{\gamma^{\beta_0}}$. Consider, as in the other case, the point $q^*=\big(q,\gamma^*\big)\in \widehat{M}_{\gamma^*}$ defined by $q_\alpha= 1/2-\gamma^*_\alpha$ if $\alpha\neq \beta_0$ and $q_{\beta_0}=1/2+\gamma^*_{\beta_0}$. Similarly, we have that $q^*\in C_{\beta_0}$, and $d(p^*,q^*)<K^{-1}\gamma^{\beta_0}_{\beta_0}$. Since $F(q^*)\in C_{\beta_0}\cap S$ and $F(p^*)\in \widehat{M}_{\gamma^{\beta_0}}$, we have that the distance between $F(p^*)$ and $F(q^*)$ is bigger than the distance from $F(p^*)$ to $C_{\beta_0}\cap \widehat{M}_{\gamma^{\beta_0}}$. Looking at the coordinate $\beta_0$ of $p^*$, we have that $\big(F(p^*)\big)_{\beta_0}<1/2-\gamma^{\beta_0}_{\beta_0}$ and the coordinate $\beta_0$ of any point in $C_{\beta_0}\cap \widehat{M}_{\gamma^{\beta_0}}$ is bigger than $1/2+\gamma^{\beta_0}_{\beta_0}$. Therefore, we obtain that

$$2\gamma^{\beta_0}_{\beta_0}<d\big(F(p^*),F(q^*)\big)<Kd(p^*,q^*)<\gamma^{\beta_0}_{\beta_0}, $$
a contradiction.
\end{proof}

\begin{remark}
This example generalizes the one in section 3 for arbitrary infinite cardinals, and its construction is considerably simpler. However, we were not able to prove that this example fails the Lipschitz RP as strongly as the example in section 3. Namely, we do not know whether a finite set can be contained in a Lipschitz retract of density character $\Lambda$.
Moreover, the first example is bi-Lipschitz equivalent to a subset of a (nonseparable) Hilbert space.  
\end{remark}
\section{Local Complementation in metric spaces}

Lemma \ref{LemmaLindenstraussLip} and Theorem \ref{analogoflocalcomplementation} in this section are the
metric analogues to Lemma 1 in \cite{Lin66} and the main Theorem in \cite{SimYos89} respectively. The result in \cite{SimYos89} states that in every Banach space $X$, each subspace $Y$ is contained in a subspace $Z$ with the same density character and such that there exists a linear norm $1$ Hahn-Banach extension operator from $Z^*$ to $X^*$ (or equivalently, $Z$ is $1$-locally complemented in $X$). We are going to obtain an analogous result for linear extensions of Lipschitz functions on metric spaces. Note however, that this result was originally proven by Heinrich and Mankiewicz using Model Theory in \cite{HeiMan82}. 

We prove that in every metric space $M$, each subset $N$ is contained in an intermediate subset $S$ with $\text{dens}(N)=\text{dens}(S)$ such that $S$ admits a linear extension operator $T\colon \text{Lip}_0(S)\rightarrow\text{Lip}_0(M)$ with $\|T\|\leq 1$. Although the statement of the results is analogous and the proof is based on the proof of the result for Banach spaces, there are several differences in the argument we write here.

The main difference between the linear and the metric settings is due to the following Lemma.
\begin{lemma}
\label{LemmaLindenstraussLip}
Let $M$ be a bounded complete metric space. Let $F\subset M$ be a finite subset of $M$, and let $k\in\mathbb{N}$ and $0<\varepsilon\leq\text{inf}_{p\neq q\in F} d(p,q)$ be given. Then there exists a finite subset $Z\subset M$ with $F\subset Z$ such that for every $\varepsilon$-separated subset $E\subset M$ with $F\subset E$ and $\text{card}(E\setminus F)\leq k$ there is a Lipschitz map $L\colon E\rightarrow Z$ with $\|L\|_\text{Lip}\leq 1+\varepsilon$ and $T(f)=f$ for all $f\in F$.
\end{lemma}
\begin{proof}
Write $R=\text{diam}(M)$ and $F=\{f_1,\dots,f_n\}$. We may assume that $\varepsilon<1$. Consider $E\subset M$ a $\varepsilon$-separated subset with $F\subset E$ and $\text{card}(E\setminus F)\leq k$. We can write this set as $E=\{f_1,\dots,f_n,p^E_1,\dots,p^E_{l_E}\}$ with $l_E\leq k$. Consider now the real valued vector:

$$a_E=(d(f_1,p^E_1),\dots,d(f_1,p^E_{l_E}),\dots,d(p^E_{l_E},p^E_1),\dots,d(p^E_{l_E},p^E_{l_E}))\in \mathbb{R}^{(n+l_E)l_E}. $$
Since $M$ has diameter $R<\infty$, the point $a_E$ belongs to $RB_{\ell_\infty^{(n+l_E)l_E}}$. Hence, if we set 

$$C=\bigsqcup_{l=1}^k RB_{\ell_\infty^{(n+l)l}}, $$
that is, the disjoint union of $RB_{\ell_\infty^{(n+l)l}}$ for $l=1,\dots,k$, then for every set $E\subset M$ with $F\subset E$ and $\text{card}(E\setminus F)\leq k$, the vector $a_E$ belongs to $C$. Since we are working with a finite disjoint union, we can endow $C$ with a metric $d_\infty$ such that $C$ is compact, the restriction of this metric to each $RB_{\ell_\infty^{(n+l)l}}$ coincides with the metric given by the supremum norm, and each $RB_{\ell_\infty^{(n+l)l}}$ is separated at least by $\varepsilon$ from its complementary in $C$. Since $C$ is compact, the subset 

$$ A_F=\{a_E\in C\colon F\subset E \text{ and card}(E\setminus F)\leq k\}\subset C$$
is totally bounded in $C$. Hence, given $\varepsilon>0$ there exist $\{E_1,\dots, E_s\}$ with $F\subset E_j$ and $\text{card}(E_j\setminus F)\leq k$ such that $A_F=\bigcup_{j=1}^s B_{\infty}(a_{E_j},\varepsilon^2)$. Set $Z=\bigcup_{j=1}^s E_j$. Let us prove that $Z$ verifies the thesis of the Lemma.

Clearly, $Z$ is finite and contains $F$. Consider any $\varepsilon$-separated subset $E\subset M$ with $F\subset E$ and $\text{card}(E\setminus F)\leq k$. There exists a $j_0\in \{1,\dots,j\}$ such that $d_\infty(a_{E},a_{E_j})\leq \varepsilon^2$ and $E_j\subset F$. Moreover, since $a_E$ and $a_{E_j}$ are closer than $\varepsilon$, they must belong to the same ball $RB_{\ell_{\infty}^{(n+l_0)l_0}}$, so $d_\infty(a_{E},a_{E_j})=\|a_{E}-a_{E_j}\|_\infty\leq \varepsilon^2$, and  $\text{card}(E_j)=\text{card}(E)=n+l_0$. Thus, we can write $E=\{f_1,\dots,f_n,p^E_1,\dots,p^E_{l_0}\}$ and $E_{j_0}=\{f_1,\dots,f_n,p^{E_{j_0}}_1,\dots,p^{E_{j_0}}_{l_0}\}$. 

Define now $L\colon E\rightarrow Z$ by $L(f)=f$ if $f\in F$, and $L(p^E_i)=p_i^{E_{j_0}}$ for $i=1,\dots,l_0$. The map $L$ verifies $L(f)=f$ for all $f\in F$ by definition, so it only remains to check that it is $(1+\varepsilon)$-Lipschitz. Since $L$ is the identity on $F$, it is sufficient to check the Lipschitz constant for pairs of points $x,y\in E$ where $x\notin F$. Then $x=p^E_{i_1}$ for some $1\leq i_1\leq l_0$. If $y=p^E_{i_2}$ for some $1\leq i_2\leq l_0$, then 

\begin{align*}
    d(L(x),L(y))&=d(p^{E_{j_0}}_{i_1},p^{E_{j_0}}_{i_2})-d(p^E_{i_1},p^E_{i_2})+d(p^E_{i_1},p^E_{i_2})\\
    &\leq \|a_{E_{j_0}}-a_E\|_{\infty}+ d(p^E_{i_1},p^E_{i_2})\leq \varepsilon\varepsilon+d(p^E_{i_1},p^E_{i_2})\leq (1+\varepsilon)d(x,y),
\end{align*}
using the fact that $E$ is $\varepsilon$-separated. If $y\in F$, then the inequality is proved similarly. We conclude that $\|L\|_\text{Lip}\leq 1+\varepsilon$, and the proof is complete.
\end{proof}

There are two main differences between this previous lemma and its analogous in \cite{Lin66} for Banach spaces. The first is that we have restricted ourselves to bounded metric spaces, and the second is that we need the set $E$ to be $\varepsilon$-separated for some $\varepsilon>0$. The boundedness problem, although it has an effect on the construction of the linear extensions in the proof of Theorem \ref{analogoflocalcomplementation} for unbounded metric spaces, is easy to work around as we will see. However, to solve the separation issue we need to alter the construction in a more meaningful way: instead of defining linear extensions to the whole metric space $M$, we will extend functions to some dense subset of $M$ that has certain separation properties. The dense subset we will use is well defined thanks to the following lemma.

\begin{lemma}
\label{denseseparatedsubsets}
Let $M$ be a complete metric space and $(F_n)_{n=1}^\infty$ be a sequence of finite subsets of $M$ with $F_n\subset F_{n+1}$ for all $n\in\mathbb{N}$, and let $(\varepsilon_n)_{n=1}^\infty$ be a decreasing sequence of positive real numbers such that $\varepsilon_n<\inf_{p\neq q\in F_n}d(p,q)$. Then there exists a sequence of sets $(D_n)_{n=1}^\infty$ with the following properties:
\begin{itemize}
    \item[(i)] $D_n\subset D_{n+1}$ for all $n\in\mathbb{N}$,
    \item[(ii)] $F_n\subset D_n$ for all $n\in\mathbb{N}$,
    \item[(iii)] $D_n\cup F_{n+k}$ is $\varepsilon_{n+k}$-separated for all $n\in\mathbb{N}$ and $k\geq 0$,
    \item[(iv)] $D=\bigcup_{n\in\mathbb{N}}D_n$ is dense in $M$,
\end{itemize}
\end{lemma}
\begin{proof}
Consider the family of sets 
$$A_1=\{D\subset M\colon F_1\subset D,~ D\cup F_{1+k} \text{ is }\varepsilon_{1+k}\text{-separated },\forall k\geq 0\}.$$
Since $(F_n)_{n=1}^\infty$ is increasing and $F_n$ is at least $\varepsilon_n$-separated, the set $F_1\in A_1$, so $A_1$ is non-empty. Consider now a chain of subsets $(C_\alpha)_{\alpha\in I}$ in $A_1$. If we define $C=\bigcup_{\alpha\in I}C_\alpha$, then $C_\alpha\subset C$ and $C\in A_1$, so it is an upper bound for the chain. By Zorn's Lemma, we can choose $D_1\in A_1$ to be maximal for the inclusion. 

Suppose we have defined $D_{n-1}$, then we define 

$$A_n=\{D\subset M\colon D_{n-1}\cup F_n\subset D,~ D\cup F_{n+k} \text{ is }\varepsilon_{n+k}\text{-separated },\forall k\geq 0\}.$$
Note that $D_{n-1}\cup F_n\in A_n$, so $A_n\neq\emptyset$. Arguing as before, we can choose a maximal set $D_n\in A_n$. Let $(D_n)_{n=1}^\infty$ be the sequence of sets obtained by this inductive process. Let us check that it verifies properties $(i)$-$(iv)$. The first three properties are clearly satisfied by definition of $A_n$. It only remains to check that $D=\bigcup_{n\in\mathbb{N}}D_n$ is dense in $M$. Suppose by contradiction that there exist an $x\in M$ and $\delta_0>0$ such that $D\cap B(x,\delta)=\emptyset$. Take an $n_0\in\mathbb{N}$ such that $\varepsilon_n<\delta$ for all $n\geq n_0$. Then clearly $D\cap B(x,\varepsilon_n)=\emptyset$ for all $n\geq n_0$, which in particular means that  $D_{n_0}\cap B(x,\varepsilon_{n_0})=\emptyset$ and $F_n\cap B(x,\varepsilon_n)=\emptyset$ for all $n\geq n_0$. But then $D_{n_0}\cup \{x\}\in A_{n_0}$, contradicting the maximality of $D_{n_0}$. We conclude that $D$ is dense, which finishes the proof.
\end{proof}

Now we can finally prove that every separable subset of a metric space is contained in an intermediate separable set that admits a linear and bounded extension of Lipschitz functions to the whole space.

\begin{theorem}
\label{analogoflocalcomplementation}
Let $M$ be a complete metric space, and let $N\subset M$ be a subset of $M$ with $0\in N$. Then there exists a subspace $S\subset M$ with $\text{dens}(N)=\text{dens} (S)$ and a linear extension operator $T\colon \text{Lip}_0(S)\rightarrow \text{Lip}_0(M)$ such that $\|T\|= 1$. In particular, every metric space has the Local CP$(\alpha,\alpha)$ for every infinite cardinal $\alpha$ smaller than the density character of $M$.
\end{theorem}
\begin{proof}
We first assume that $N$ is separable.

We are going to find a linear extension operator $T\colon \text{Lip}_0(D\cap S)\rightarrow \text{Lip}_0(D)$ with $\|T\|=1$, where $D$ is a dense subset of $M$ such that $D\cap S$ is dense in $S$. This will suffice to prove the result, since it is known that we can extend linearly Lipschitz functions from dense subsets preserving the Lipschitz constant, and this extension is unique by continuity (check Proposition 1.6 in \cite{Wea18Book}, for instance).

Let $(p_n)_{n=1}^\infty$ be a dense sequence in $N$. For $n=0$, put $S_0=\{0\}$. Inductively, suppose we have defined $S_{n-1}$, which is finite. Put $F_n=S_{n-1}\cup\{p_n\}$ as a finite set, $\theta_{n}=\inf_{p\neq q\in F_{n}} d(p,q)$ as the separation of said set, and $r_n=\text{rad}(F_n)$ its radius. Set $\varepsilon_n=\min\{1/n,\theta_{n}\}$ and $R_n=\max\{r_n,n\}$. We choose $S_n$ to be the set $Z$ given by Lemma \ref{LemmaLindenstraussLip} applied to $M\cap B(0,R_n)$, which is bounded, with $F=F_n$, $k=n$ and $\varepsilon=\varepsilon_n$. Set $S=\overline{\bigcup_{n\in\mathbb{N}}S_{n}}$. Then clearly $S$ is separable and contains $N$.

Let $(D_n)_{n=1}^\infty$ be the increasing sequence of sets given by Lemma \ref{denseseparatedsubsets} applied to $(F_n)_{n=1}^\infty$ and $(\varepsilon_n)_{n=1}^\infty$. Notice that if $D=\bigcup_{n\in\mathbb{N}}D_n$, then $D$ is dense in $M$ by the Lemma, and $D\cap S$ is dense in $S$ because $S_n\subset D_{n+1}$ for all $n\in\mathbb{N}$.

Fix $n\in\mathbb{N}$, and define the family of subsets:
$$I_{n}=\{E\subset D_n\cap B(0,R_n)\colon F_n\subset E,~ E\text{ is }\varepsilon_{n}\text{-separated, and } \text{card}(E\setminus F_n)\leq n\}. $$
Note that if $E\subset D_n$, the condition that $E$ is $\varepsilon_n$-separated is redundant, but we state it for clarity. Indeed, now it is clear that if $E\in I_n$, then there exists a Lipschitz map $L_E\colon E\rightarrow S_{n}$ with $(L_E)_{|F_{n}}=\text{Id}_{F_{n}}$ and $\|L_E\|_\text{Lip}\leq 1+\varepsilon_n$. 

Let 
$$I=\bigcup_{n\in\mathbb{N}}I_{n}. $$
Notice that if $E_1\in I_{n_1}$, $E_2\in I_{n_2}$, then all three of $E_1,E_2$ and $F_{n_1+n_2}$ are contained in $D_{n_1+n_2}$. Hence $E_0=E_1\cup E_2\cup F_{n_1+n_2}$ is $\varepsilon_{n_1+n_2}$-separated, and moreover $E_0\setminus F_{n_1+n_2}\subset (E_1\setminus F_{n_1})\cup (E_2\setminus F_{n_2})$, which means that $\text{card}(E_0\setminus F_{n_1+n_2})\leq n_1+n_2$. Therefore $E_0\in I_{n_1+n_2}\subset I$. Thus $I$, with the order given by inclusion, is a directed set.

For a set $E\in I$, consider $I_{E}:=\{Z\in I\colon E\subset Z\}$, which is a subset of $I$. Since $I$ is directed, the family $B=\{I_E\}_{E\in I}$ is closed for finite intersections and $\emptyset\notin B$, so it is the subbase of a filter on $\mathcal{P}(I)$. Let $U$ be an ultrafilter that extends this filter. For any $p\in D$, there exists a minimum $n_p\in\mathbb{N}$ such that $p\in D_{n_p}\cap B(0,n_p)$, so the set $I_p:=\{Z\in I\colon p\in Z\}$ belongs to $B\subset U$, since it can be written as $I_p=I_{E_p}=\{Z\in I\colon E_p\subset Z\}$, where $E_p=\{p\}\cup F_{n_p}$, which is a member of $I_{n_p}\subset I$.

Also note that for every $n\in\mathbb{N}$, the set $I_n$ can be written as $I_{F_n}$, so $I_n\in U$ as well.

For each $E\in I$ define $n(E):=\max\{n\in\mathbb{N}\colon E\in I_{n}\}$ which exists since $E$ is finite. Since $E\in I_{n(E)}$, there exists a Lipschitz map $L_E\colon E\rightarrow S_{n(E)}$ with $(L_E)_{|F_{n(E)}}=\text{Id}_{F_{n(E)}}$ and $\|L_E\|_\text{Lip}\leq 1+\varepsilon_{n(E)}$. We can extend each $L_E$ to a non-Lipschitz function defined on the dense subset $D=\bigcup_{n\in\mathbb{N}}D_n$ by simply defining $\widetilde{L}_E\colon M\rightarrow S_{n(E)}$ as 
$$
\widetilde{L}_E(p)=
\begin{cases}
L_E(p),&\text{ if } p\in E,\\
0, &\text{ if }p\in D\setminus E.
\end{cases}
$$

Finally, the linear extension operator $T\colon \text{Lip}_0(S)\rightarrow \text{Lip}_0(D)$ is defined for each $f\in \text{Lip}_0(D\cap S)$ by:
$$(Tf)(p)=\lim_{U} f(\widetilde{L}_E(p)), \qquad p\in D.$$
This limit exists since $U$ is an ultrafilter and $\bigcup_{E\in I}f(\widetilde{L}_E(p))$ is relatively compact in $\mathbb{R}$ for all $p\in D$.

Let us check that $T$ is a linear extension operator with norm $1$. First, it follows that $T$ is linear by the linearity of the limit with respect to an ultrafilter. 

By the definition of limit with respect to an ultrafilter and by definition of $S$, to prove that $T$ is an extension operator from $\text{Lip}_0(D\cap S)$, it is enough to show that given $f\in\text{Lip}_0(S)$, $n_0\in\mathbb{N}$ and  $p_0\in S_{n_0}$, there is a set $I_{0}\in U$ such that $f(\widetilde{L}_E)(p)=f(p)$ for all $E\in I_{0}$. Consider $I_0=I_{F_{n_0+1}}\in U$, and note that $F_{n_0+1}\in I$ and verifies $S_{n_0}\subset F_{n_0+1}$. For every $E\in I_0$, we have that $F_{n_0+1}\subset E$, so $(L_E)_{|F_{n_0+1}}=\text{Id}_{F_{n_0+1}}$. In particular, $L_E(p)=p$ for all $p\in S_{n_0}$, so $f(\widetilde{L}_E)(p)=f(p)$ as desired. 

It only remains to prove that $\|T\|=1$. Again, it suffices to show that for any pair of points $x,y\in D$ and any $\delta>0$, there exists $I_1\in U$ such that for all $E\in I_1$, the inequality $|f(\widetilde{L}_E(x))-f(\widetilde{L}_E(y))|\leq (1+\delta)\|f\|_\text{Lip}d(x,y)$ holds. Consider $n_0\in\mathbb{N}$ such that $\varepsilon_{n_0}<\delta$, and set $I_1=I_x\cap I_y\cap I_{n_0}\in U$. If $E\in I_1$, then $x,y\in E$, so $\widetilde{L}_E(x)=L_E(x)$ and $\widetilde{L}_E(y)=L_E(y)$. Therefore, since $L_E$ is $(1+\varepsilon_{n_0})$-Lipschitz in $E$, we have that 

\begin{align*}
    |f(\widetilde{L}_E(x))-f(\widetilde{L}_E(y))|&=\|f\|_\text{Lip}|\widetilde{L}_E(x)-\widetilde{L}_E(y)|= \|f\|_\text{Lip}|L_E(x)-L_E(y)|\\
    &\leq (1+\varepsilon_{n_0})\|f\|_\text{Lip}d(x,y)\leq (1+\delta)\|f\|_\text{Lip}d(x,y),
\end{align*}
and the proof for the separable case is complete.

For the general case we use transfinite induction. Suppose that $N$ is nonseparable, let $\lambda=\text{dens}(N)$, and suppose that we have
proved the result for every cardinal $\alpha$ with $\omega_0\leq \alpha <\lambda$. Choose $\{p_\alpha\}_{\alpha< \lambda}$. Since $\{p_\alpha\}_{\alpha<\omega_0}$ is countable, there exists a separable subset $S_{\omega_0}\subset M$ with $p_\alpha\in S_{\omega_0}$ for all $\alpha<\omega_0$ and a norm $1$ linear extension operator $T_{\omega_0}\colon \text{Lip}_0(S_{\omega_0})\rightarrow\text{Lip}_0(M)$. 
Similarly, for $\omega_0<\alpha<\lambda$, we can find a subset $S_\alpha\subset M$ containing $\bigcup_{\omega_0<\beta<\alpha}S_\beta\cup \{p_\alpha\}$ with $\text{dens}(S_\alpha)\leq \alpha$ and a norm $1$ linear extension operator $T_{\omega_0}\colon \text{Lip}_0(S_{\omega_0})\rightarrow\text{Lip}_0(M)$. Set

$$S=\overline{\bigcup_{\omega_0<\alpha<\lambda}S_{\alpha}}. $$
We have that $N\subset S$ and $\text{dens}(S)=\lambda$. For any $\omega_0<\alpha<\lambda$, consider the linear map $R_\alpha\colon \text{Lip}_0(S)\rightarrow \text{Lip}_0(S_\alpha)$ given by the restriction to $S_\alpha$. We have  that $E_\alpha=T_\alpha R_\alpha\colon \text{Lip}_0(S)\rightarrow \text{Lip}_0(M)$ is a bounded linear map with $\|E_\alpha\|\leq 1$. Therefore, for any $f\in \text{Lip}_0(S)$, the set $\{E_\alpha f\colon \omega_0\leq \alpha<\lambda\}$ is bounded in $\text{Lip}_0(M)$, and thus relatively compact for the weak$^*$ topology. 

Let $U$ be a non-principal ultrafilter on $\{\alpha\colon \omega_0\leq \alpha<\mu\}$. Then we define $E\colon \text{Lip}_0(S)\rightarrow \text{Lip}_0(M)$ by 

$$Ef:=w^*\text{-}\lim_{U} E_\alpha f, $$
which is well defined by the discussed compactness and the fact that $U$ is an ultrafilter. It is straightforward to check that $E$ is a linear extension operator with $\|E\|\leq 1$. Note that $E$ can also be obtained as the weak$^*$ limit of a subnet of $E_\alpha$ using  the compactness of $B_{\text{Lip}_0(M)}$ in the weak$^*$ topology.
\end{proof}

As we mentioned in the introduction and as suggested by our definition of local complementation in metric spaces, the existence of a linear extension operator for Lipschitz functions in the metric space setting is analogous to the local complementation property in Banach spaces. We finish this section making this analogy clearer, summarizing the relationship between local complementation and linear extension operators in Banach spaces. 

Let $X$ be a Banach space and $\lambda\geq 1$. We say that a subspace $Z\subset X$ is \emph{$\lambda$-locally complemented in $X$} if for every finite-dimensional subspace $F\subset X$ and every $\varepsilon>0$ there exists a linear operator $T\colon F\rightarrow Y$ with $\|T\|\leq \lambda$ such that $\|Tf-f\|<\varepsilon\|f\|$ for all $f\in Y\cap F$. By the Principle of Local Reflexivity, every Banach space is $1$-locally complemented in its bidual.
There are several well known equivalent formulations of this property in the literature (see \cite{Kal84}, \cite{Fak72}, \cite{GonMar15}), some of which concern linear extensions of certain classes of functions. 

\begin{proposition}
\label{equivalencesoflocalcompl}
Let $X$ be a Banach space, $Z\subset X$ a linear subspace, and $\lambda\geq 1$. The following statements are equivalent:
\begin{itemize}
    \item[(1)] $Z$ is $\lambda$-locally complemented in $X$.
    \item[(2)] There exists a linear projection $P\colon X^*\rightarrow Z^{\perp}$ such that $\|\text{Id}_{X^*}-P\|\leq \lambda$.
    \item[(3)] $Z^{**}$ is $\lambda$-complemented in $X^{**}$ in its natural embedding.
    \item[(4)] $Z$ has the Compact Extension Property in $X$, i.e.: for every Banach space $Y$ and every linear compact operator $K\colon Z\rightarrow Y$, there exists a compact operator $\widehat{K}\colon X\rightarrow Y$ that extends $K$ and such that $\|\widehat{K}\|\leq \lambda\|K\|$. 
    \item[(5)] There exists a linear operator $T\colon X\rightarrow Z^{**}$ such that $\|T\|\leq \lambda$ and $T$ is the identity on $Z$.
    \item[(6)] There exists an ultrafilter $U$ and a linear operator $T\colon X\rightarrow (Z)_{U}$ such that $\|T\|\leq \lambda$ and $T$ is the identity on $Z$ (where $(Z)_U$ denotes the ultraproduct with respect to $U$ and identifying $Z$ with its natural copy in $(Z)_U$).
\end{itemize}
\end{proposition}

A usual technique to find a linear operator that extends functions from a Banach space $X$ to its bidual, preserving some specific local property of these functions, involves extending functions from $X$ to an ultraproduct $(X)_U$ with a similar process as the one we used in Theorem \ref{analogoflocalcomplementation}, and then using the principle of local reflexivity to show that $X^{**}$ can be seen as a subspace of $(X)_U$ that contains $X$ (we refer to \cite{HajMic14} for a detailed study of these extensions). Thanks to statement (6) in the previous proposition, we can use this technique to construct linear extension operators from locally complemented spaces, which generalizes the bidual extensions. In particular,
we obtain a generalization of the Aron-Berner polynomial (and analytic)  extensions in \cite{AroBer78}, as well as the bidual extensions of $k$-times uniformly differentiable functions introduced in \cite{ChoHajLee13} and further studied in  \cite{HajMic14}.

For the next proposition, denote by $X^{\mathbb{R}}$ the set of real valued functions defined on the Banach space $X$. Then we have:
\begin{proposition}
Let $X$ be a Banach space, $\lambda\geq 1$. If a subspace $Z\subset X$ is $\lambda$-locally complemented, then there exists a linear extension map $E\colon Z^{\mathbb{R}}\rightarrow X^{\mathbb{R}}$ such that:

\begin{itemize}
    \item[(1)] If $f\in Z^*\subset Z^{\mathbb{R}}$, then $E(f)\in X^*$ and $\|E(f)\|\leq \lambda \|f\|$. In particular, there exists a linear extension operator $T\colon Z^*\rightarrow X^*$ with $\|T\|\leq \lambda$.
    \item[(2)] If $f\in \text{Lip}_0(Z)\subset Z^{\mathbb{R}}$, then $E(f)\in \text{Lip}_0(X)$ and $\|E(f)\|_{\text{Lip}}\leq \lambda \|f\|_{\text{Lip}}$. In particular, there exists a linear extension operator $T\colon \text{Lip}_0(Z)\rightarrow \text{Lip}_0(X)$ with $\|T\|\leq\lambda$.
    \item[(3)] If $f$ is uniformly continuous in $X$ with modulus of continuity $\omega$, then $E(f)$ is uniformly continuous with modulus of continuity less or equal than $\lambda\omega$. 
    \item[(4)] If $f$ is a continuous polynomial in $Z$ of degree $n$, then $E(f)$ is a continuous polynomial in $X$ of degree $n$, and $\|E(f)\|\leq \lambda\|f\|$.
    \item[(5)] If $f$ is $k$-times uniformly Fr\'echet differentiable in $Z$, for some $k\ge1$, then $E(f)$ is $k$-times uniformly Fr\'echet differentiable in $X$, and the derivative coincides at any point in $Z$.
\end{itemize}
\end{proposition}
Note also that conditions $(1)$ and $(2)$ in the previous proposition are actually equivalent to the fact that $Z$ is $\lambda$-locally complemented in $X$.

\printbibliography
\end{document}